\documentclass{article}   

\usepackage{graphicx, amssymb, amsmath}
\usepackage{amsthm}
\usepackage{indentfirst}
\usepackage{mathrsfs}
\usepackage{amsfonts}
\usepackage{color}
\usepackage{algorithm}
\usepackage{algorithmic}
\usepackage{upgreek}
\usepackage{subfigure}

\usepackage[colorlinks,linkcolor=blue,anchorcolor=blue,citecolor=blue]{hyperref}  

\usepackage{booktabs}    
\usepackage{multirow}
\usepackage{makecell}

\usepackage{caption}
\usepackage{titlesec}

\usepackage[numbers]{natbib}  

\def\k{\kern .5em}
\def\er{\kern .2em}

\begin{document}
\author{}

\newcommand{\be}{\begin{equation}}
\newcommand{\ee}{\end{equation}}
\newcommand{\ba}{\begin{array}}
\newcommand{\ea}{\end{array}}
\newcommand{\beas}{\begin{eqnarray*}}
\newcommand{\eeas}{\end{eqnarray*}}
\newcommand{\bea}{\begin{eqnarray}}
\newcommand{\eea}{\end{eqnarray}}
\newcommand{\ome}{\Omega}

\newtheorem{theorem}{Theorem}[section]
\newtheorem{lemma}{Lemma}[section]
\newtheorem{remark}{Remark}[section]
\newtheorem{proposition}{Proposition}[section]
\newtheorem{definition}{Definition}[section]
\newtheorem{corollary}{Corollary}[section]

\newcommand{\tabincell}[2]{\begin{tabular}{@{}#1@{}}#2\end{tabular}}

\newtheorem{theo}{Theorem}[section]
\newtheorem{lemm}{Lemma}[section]
\newcommand{\blem}{\begin{lemma}}
\newcommand{\elem}{\end{lemma}}
\newcommand{\bthe}{\begin{theorem}}
\newcommand{\ethe}{\end{theorem}}
\newtheorem{prop}{Proposition}[section]
\newcommand{\bprop}{\begin{proposition}}
\newcommand{\eprop}{\end{proposition}}
\newtheorem{defi}{Definition}[section]
\newtheorem{coro}{Corollary}[section]
\newtheorem{algo}{Algorithm}[section]
\newtheorem{rema}{Remark}[section]
\newtheorem{property}{Property}[section]
\newtheorem{assu}{Assumption}[section]
\newtheorem{exam}{Example}[section]

\renewcommand{\theequation}{\arabic{section}.\arabic{equation}}
\renewcommand{\thetheorem}{\arabic{section}.\arabic{theorem}}
\renewcommand{\thelemma}{\arabic{section}.\arabic{lemma}}
\renewcommand{\theproposition}{\arabic{section}.\arabic{proposition}}
\renewcommand{\thedefinition}{\arabic{section}.\arabic{definition}}
\renewcommand{\thecorollary}{\arabic{section}.\arabic{corollary}}
\renewcommand{\thealgorithm}{\arabic{section}.\arabic{algorithm}}
\newcommand{\lan}{\langle}
\newcommand{\curl}{{\bf curl \;}}
\newcommand{\rot}{{\rm curl}}
\newcommand{\grad}{{\bf grad \;}}
\newcommand{\dvg}{{\rm div \,}}
\newcommand{\ran}{\rangle}
\newcommand{\bR}{\mbox{\bf R}}
\newcommand{\bRn}{{\bf R}^3}
\newcommand{\Coinf}{C_0^{\infty}}
\newcommand{\disp}{\displaystyle}
\newcommand{\ra}{\rightarrow}
\newcommand{\Ra}{\Rightarrow}
\newcommand{\ud}{u_{\delta}}
\newcommand{\Ed}{E_{\delta}}
\newcommand{\Hd}{H_{\delta}}
\newcommand\varep{\varepsilon}


 \title{A Decoupling Two-grid Method for the Time-dependent Poisson-Nernst-Planck Equations}

\author{ Ruigang Shen$^1$
 \and Shi Shu$^2$
 \and Ying Yang$^{3,*}$
 \and Benzhuo Lu$^4$ }
 \footnotetext[1]
 { School of  Mathematics and Computational Science, Xiangtan University, Xiangtan  411105, Hunan, P.R. China. E-mail: rgshen@hotmail.com }
 \footnotetext[2]
 { Hunan Key Laboratory for Computation and Simulation in Science and Engineering, Xiangtan University, Xiangtan 411105, Hunan, P.R. China.
 E-mail: shushi@xtu.edu.cn  }
 \footnotetext[3]
 {$^{,*}$\textbf{Corresponding author.} School of  Mathematics and Computational Science, Guangxi Colleges and Universities Key Laboratory of Data Analysis and Computation, Guangxi Key Laboratory of Cryptography and information Security, Guilin University of Electronic Technology, Guilin 541004, Guangxi, P.R. China. E-mail: yangying@lsec.cc.ac.cn }
 \footnotetext[4]
 {Institute of Computational Mathematics and Scientific/Engineering Computing, the National Center for Mathematics and Interdisciplinary Sciences, Academy of Mathematics and Systems Science, Chinese Academy of Sciences, Beijing 100190, P.R. China. E-mail: bzlu@lsec.cc.ac.cn}

\date{}
\maketitle


\noindent {\bf Abstract}
We study a two-grid strategy for decoupling the time-dependent Poisson-Nernst-Planck equations describing the mass concentration of ions and the electrostatic potential. The computational system is decoupled to smaller systems by using coarse space solutions  at each time level, which can speed up the solution process compared with the finite element method combined with the Gummel iteration. We derive the optimal error estimates in $L^2$ norm for both semi- and fully discrete finite element approximations. Based on the a priori error estimates, the error estimates in $H^1$ norm are presented for the two-grid algorithm. The theoretical results indicate this decoupling method can retain the same accuracy as the finite element method. Numerical experiments including the Poisson-Nernst-Planck equations for an ion channel show the efficiency and effectiveness of the decoupling two-grid method.\\




\section{Introduction}{\label{sec1}
 \noindent
 \par In this paper, we consider the following time-dependent Poisson-Nernst-Planck (PNP) equations
 \begin{align}  \label{pnpt1}
 \left\{ \begin{array} {rcl}
   {\partial _t}{p^i} - \nabla\cdot(\nabla {p^i} + {q^i}{p^i}\nabla \phi ) &=& F_i,~i = 1,2,     \\
   \\
  - \Delta \phi  - \sum\limits_{i = 1}^2 {{q^i}} {p^i} &=& F_3,   
  \end{array} \right.
 \end{align}
 for $ x\in \Omega$ and $t\in [0,T]$, where $\Omega$ is a bounded Lipschitz domain in $\mathbb{R}^d ~(d = 2, 3)$ and $\partial_t=\partial/\partial_t$. The index $i$ represents different ionic species, $p^i$ is the concentration of the $i$th ionic species with charge $q^i$, $\phi$ is the electrostatic potential and $F_i~(i=1,2,3)$ are the reaction terms. Denote the initial concentrations and potential by $(p^{i,0},\phi^0), \ i=1,2$. For simplicity, we employ the following homogeneous Dirichlet boundary conditions
 \begin{align}  \label{pnpt1-bd}
  p^1=p^2=\phi=0,~\mbox{on} ~\partial\Omega \times(0,T].
 \end{align}
 The classic PNP system was first proposed by W. Nernst \cite{W. Nernst 89} and M. Planck \cite{M. Planck 90}. It mainly describes the mass concentration of ions $p^i:\Omega\times (0,T]\rightarrow \mathbb{R}_0^+$ and the electrostatic potential $\phi:\Omega\times (0,T]\rightarrow \mathbb{R}$.
 As a continuum electrodiffusion model, PNP equations play an important role in the electrodiffusion reaction process. PNP equations couple the ion concentration distributions with the electrostatic potential which provide an ideal mean-field for describing this process \cite{R. Eisenberg 1993,B.Z.Lu-Holst 2010}. They have been widely used to study the ion channels and nanopores etc. \cite{R. D. Coalson 2005,U. Hollerbach 2000,A. Singer 2009,S. Xu 2014,U. Hollerbach 2001}.

 Since the strong nonlinearity and coupling of the PNP system, in general, it is difficult to find the analytic solution of PNP equations. Therefore, there appears many  numerical methods for solving PNP equations, including finite difference method, finite volume method and finite element method, etc. Finite difference method has been widely used to solve the PNP equations \cite{Q. Zheng 2011, A. Flavell 2013,D. He 2016, H. Liu 2014,M. Mirzadeh 2014},
   but the accuracy is not so good when it is applied to the biomolecular models with highly irregular surfaces. Finite volume method, which focuses on avoiding the disadvantage of finite difference method, was then applied to solve the PNP equations in irregular domains, but it is not easy to achieve the high accuracy owing to the difficulty of the design of high-order control volume \cite{S. R. Mathur 2009,J. Wu 2002}. Finite element method (FEM) has more flexibility and adaptability in irregular regions, which  has shown the efficiency and effectiveness of dealing with PNP equations \cite{Y. H. Song 2004,Y. C. Zhou 2008,B.Z.Lu-Holst 2010}.

 In contrast to amount of work on the numerical computations of PNP equations, the work of mathematical analysis of PNP equations seems limited, especially for finite element method. The existence and uniqueness of the finite element approximation for the time-dependent PNP equations are shown in \cite{A. Prohl 09}.  Recently, Yang and Lu \cite{Y. Yang 2013} presented an error analysis of the finite element method for a type of steady-state PNP equations modeling the electrodiffusion of ions in a solvated biomolecular system, in which the error estimates for the potential and concentration in $H^1$ norm depend on the $L^2$ error of the concentration. Sun et al. \cite{Y.Z. Sun 16} analyzed a fully implicit nonlinear Crank-Nicolson scheme of the finite element method for the PNP equations, where an optimal $H^1$ norm error estimate is obtained for both the ion concentration and electrostatic potential. They also presented a $L^2$ norm error estimate which is only sub-optimal for linear finite element approximations. Soon afterwards, Gao and He \cite{H. D. Gao 17} obtained an optimal $L^2$ error estimate with linear finite element approximations for a linearized backward Euler scheme. It is shown that this linearized scheme can preserve mass consevation and energy decay.
 In this paper, we shall present an optimal $L^2$ error estimate for the classic backward Euler scheme. Compared with the scheme in \cite{H. D. Gao 17}, this one is fully implicit nonlinear. It is considered that this implicit nonlinear scheme could preserve most of the properties of the PNP equations and has been commonly used in the computation of the PNP system \cite{B.Z.Lu-Holst 2010,B.Z.Lu 2011,S. Xu 2014}.
 The optimal error estimates in $L^2$ norm are obtained for both semi- and fully discrete finite element approximations.
  These results shall be used in the error analysis of the main algorithm of the paper.

 The PNP equations are a type of strong coupled system. Since the system consists of more than two partial differential equations, generally speaking, it is more convenient to solve it by using a decoupling method than solving it directly in application for large scale problems. Decoupling methods, by which the coupled problems can be separated into single subproblems, have some appealing features. For example, the existed computing resources are more flexibly applied to solving each subproblem separately, and the numerical implementation is more easy and efficient. The main decoupling methods used currently for solving PNP equations is the Gummel iteration \cite{M.Burger 2009,J.W.Jerome 1987,B.Z.Lu 2011}.
 For example, consider the following system coupled by two equations
 \begin{align}
  \left\{ \begin{array} {rcl}
    F(u_1, u_2) = 0,  \\
     \\
    G(u_1, u_2) = 0.
   \end{array}   \right.
 \end{align}
 The Gummel iteration for the above system could be: given $u_2^0$, for $k\ge0$, find $(u_1^{k+1},u_2^{k+1})$ such that
  \begin{align}\label{gummel}
  \left\{ \begin{array} {rcl}
    F(u_1^{k+1}, u_2^{k}) = 0,  \\
     \\
    G(u_1^{k+1}, u_2^{k+1}) = 0,
   \end{array}   \right.
 \end{align}
 until the error between the $(k+1)$th solution and $k$th solution is less than the tolerance. However, it converges slowly even diverges  if the discretized system of the PNP equations is a large scale problem.

 We note that two grid method is also one of decoupling methods which has been applied successfully to some coupled systems such as the Schr\"{o}dinger equation arising from quantum mechanics \cite{J. Jin 2006} and Stoke-Darcy model for coupling fluid flow with porous media flow \cite{M. Mu 2007,M. Cai 2009}. Two-grid method, proposed originally by Xu \cite{J. Xu 1992} in 1992, was designed for dealing with nonselfadjoint or indefinite problems and has a variety of application to solving many problems, such as the nonlinear reaction-diffusion equation \cite{Y. Liu 2015, L. Wu 2015} and the nonlinear parabolic equation \cite{Y. P. Chen 2013} etc. As a decoupling method for the coupled equations, the procedure of the two grid method may be different from that for a single partial different equation mentioned above, but has the similar idea that a coarse space solution is chosen as a reliable approximation to the fine space solution. In the two-grid algorithms designed in this paper for decoupling the time-dependent PNP equations, since we can use an appropriate coarse space solution as a reliable approximation to the fine space solution, the iteration between the equations solving individually can be avoided on the fine space, while it may requires lots of iterations for the Gummel method (\ref{gummel}) if an inappropriate initial value is used. Moreover, since the two-grid method is based on the finite element method, the numerical implement of the decoupling process is easy if the finite element method is used to solve PNP equations. These are the main reasons that we consider the two-grid method to deal with PNP equations among many decoupling methods.

 In this paper, we propose and analyze the two-grid algorithm for time-dependent PNP equations in a fully discrete scheme. Since PNP equations are different from the coupled models mentioned above, the design and analysis of the two-grid method can not directly follow the existed work. The error estimates in $H^1$ norm are obtained for both the concentration and potential. The theoretical results show that if the mesh size $H$ and $h$ satisfy some requirement (for example $H=\mathcal{O}(h^{\frac{1}{2}})$ with linear finite elements), then the two-grid method can retain the same accuracy as the conventional finite element method. In addition, some numerical examples including an ion channel problem are shown to verify the theoretic results. The CPU time cost shows the validity and efficiency of the two-grid method for PNP equations.

 The rest of this paper is organized as follows. In Section \ref{sec2}, we introduce some notations and the weak formulations of the PNP system. The projection operators and some useful estimates are also given in this section. In Section \ref{sec3}, we show the optimal $L^2$ error estimates of the standard finite element method for both semi- and fully discrete schemes. The two-grid method and some error analysis are presented in section \ref{sec4}. Numerical experiments are reported in Section \ref{sec-num-res} to show the effectiveness of the proposed method. The conclusion is presented in section \ref{sec-conclusion}.

\setcounter{equation}{0}
\section{Weak Formulation and Projection Operators} \label{sec2}

\noindent
 \par In this section, we shall present the variational forms of PNP system \eqref{pnpt1}-\eqref{pnpt1-bd} and some projection estimates which shall be used in our analysis.

 First, we clarify the standard notations for Sobolev spaces $W^{s,p}(\Omega)$ and their associated norms and seminorms, see, e.g., \cite{R. A. Adams 75, S.C.Brenner 2002}. For $p=2$, we denote $W^{s,2}(\Omega )=H^s(\Omega)$, $ H_0^1(\Omega)=\{v|v\in H^1(\Omega):v|_{\partial \Omega}=0\} $, $\|\cdot\|_{s,p,\Omega}=\|\cdot\|_{W^{s,p}(\Omega)}$ with the expression that $\|\cdot\|$ and $(\cdot,\cdot)$ denote the norm and inner product in $L^2$, and $\|\cdot\|_{0,\infty}=\|\cdot\|_{L^\infty}$.

  Let $\mathcal{T}_h$ be a quasi-uniform partition of $\Omega = \cup_e \Omega_e$, and the mesh size $h = \mbox{max}_{{\Omega_e}\in\mathcal{T}_h} \{\mbox{diam}~\Omega_e\}$. Then for a given partition $\mathcal{T}_h$, we define $V_h^r$ as the $r$-th order finite element subspace of $H_0^1(\Omega)$ as follows
 \begin{eqnarray}
  V_h^r=\{v \in H^1(\Omega): v|_{\partial\Omega}=0~\mbox{and}~v|_e \in P_r(e), \forall e\in \mathcal{T}_h\},
 \end{eqnarray}
 where $P_r(e)$ is the space of polynomial with degree $r$.

 The weak formulation of \eqref{pnpt1}-\eqref{pnpt1-bd} is reads: find $p^i\in L^2\left(0,T;H_0^1(\Omega)\right) \ \cap  \ L^\infty\left(0,T;L^\infty(\Omega)\right),i=1,2$, and $\phi(t)\in H_0^1(\Omega)$ such that
 \begin{align} \label{pnpw1}
 (\partial_t p^i,v)+(\nabla p^i,\nabla v)+(q^ip^i\nabla\phi,\nabla v)
 &=(F_i,v), \quad~\forall v\in H_0^1(\Omega), \\  \label{pnpw2}
 (\nabla \phi,\nabla w)-\sum\limits_{i = 1}^2 q^i(p^i,w)
 &=(F_3,w), \quad \forall w\in H_0^1(\Omega).
 \end{align}

 The corresponding semi-discretization to \eqref{pnpw1}-\eqref{pnpw2} is defined as follows: find $(p_h^i,\phi_h)\in [V_h^r]^3,i=1,2$, such that
 \begin{align} \label{pwh1}
  (\partial_t p_h^i,v_h)+(\nabla p_h^i,\nabla v_h)
  +(q^ip_h^i\nabla\phi_h,\nabla v_h)
 &=(F_i,v_h), \quad~\forall v_h\in V_h^r, \\ \label{pwh2}
 (\nabla \phi_h,\nabla w_h)-\sum\limits_{i = 1}^2 q^i(p_h^i,w_h)
 &=(F_3,w_h),\quad \forall w_h\in V_h^r,
 \end{align}
 with the initial condition $(p_h^{i,0},\phi_h^0)$ is an approximation of $(p^{i,0},\phi^0)$ and the Dirichlet boundary condition $p_h^i=\phi_h=0$ on $\partial \Omega$.

 In order to get the full discretization  of the system \eqref{pnpw1}-\eqref{pnpw2}, we first define a uniform partition
 $0=t^0<t^1<\cdots<t^N=T$ with time step size $\tau=\frac{T}{N}$ and $t^n=n\tau, n\in\mathbb{Z}$. For any function $u$, denote by
 $$u^n = u(x,t^n),$$
 and
 $$D_\tau u^{n+1} =\frac{u^{n+1}-u^{n}} { \tau} , ~\mbox{for}~ n=0,1,2,\cdots,N-1.$$

 Then the backward Euler full discretization scheme of the system \eqref{pnpw1}-\eqref{pnpw2} is : given $(P_h^{i,n},\Phi_h^n)\in [V_h^r]^3,\ i=1,2$, find $(P_h^{i,n+1}, \Phi_h^{n+1})\in [V_h^r]^3$, such that
 \begin{align} \label{fph1}
 (D_\tau P_h^{i,n+1},v_h)+(\nabla P_h^{i,n+1},\nabla v_h)+(q^iP_h^{i,n+1}\nabla\Phi_h^{n+1},\nabla v_h)
 &= (F_i^{n+1},v_h), \quad~\forall v_h\in V_h^r, \\ \label{fph2}
 (\nabla\Phi_h^{n+1},\nabla w_h)-\sum\limits_{i = 1}^2 q^i(P_h^{i,n+1},w_h)
 &= (F_3^{n+1},w_h),\quad\forall w_h\in V_h^r.
 \end{align}

 The well-posedness and stability of the solutions  to the the schemes \eqref{fph1}-\eqref{fph2} have been presented  in \cite{A. Prohl 09}.
 In the rest part of this paper, we assume that the exact solution of the
 PNP equations \eqref{pnpt1} exists and satisfies the following regularity assumptions
 \begin{eqnarray} \label{Reg_assumption}
 \left\{ \begin{array} {l}
  \|p^i\|_{L^{\infty}(0,T; H^{r+1}\cap W^{1,\infty}(\Omega))} +\|p^i_t\|_{L^{\infty}(0,T; H^{r+1}\cap W^{1,\infty}(\Omega))}+\|p^i_{tt}\|_{L^{\infty}(0,T; H^{r+1}\cap W^{1,\infty}(\Omega))}\le C,  \\
  \\
  \|\phi\|_{L^{\infty}(0,T;W^{r+1,\infty}(\Omega))} +\|\phi_t\|_{L^{\infty}(0,T;W^{r+1,\infty}(\Omega))}
  +\|\phi_{tt}\|_{L^{\infty}(0,T; W^{r+1,\infty}(\Omega))} \le C.
  \end{array} \right.
 \end{eqnarray}

 To present the error estimates in this paper, for given $t\in[0,T]$, we define $R_h: H_0^1(\Omega)\rightarrow V_h^r$ to be a Ritz projection operator by
 \begin{eqnarray} \label{Rhprojetion-1}
  \big(\nabla(R_hp^i-p^i),\nabla v_{h}\big)
      +\big(q^ip^i\nabla(R_h\phi-\phi),\nabla v_{h}\big)=0, ~~\forall ~v_{h}\in V_h^r, \\ \label{Rhprojetion-2}
  \big(\nabla(R_h\phi-\phi),\nabla w_h\big)
      -\big(\sum\limits_{i=1}^2q^i(R_hp^i-p^i),w_h\big)=0,~~\forall ~w_{h}\in V_h^r.
 \end{eqnarray}

 Particularly, the similar definition of the projection operator $R_h$ can be found in \cite{H. D. Gao 17}. At the initial step in \eqref{fph1}-\eqref{fph2}, we take the initial value $p_h^{i,0} = R_hp^{i,0}$.

 We define the projection error by
 $$ \theta_{p^i}=R_hp^i-p^i, ~ \theta_\phi=R_h\phi-\phi.$$
 Then, by standard finite element theory and the regularity assumption \eqref{Reg_assumption}, we have
 \begin{eqnarray} \label{projec-estimate1}
  \|\theta_{p^i}\|+h\|\theta_{p^i}\|_1 \le C h^{r+1}, \\ \label{projec-estimate2}
  \|\theta_\phi\|+h\|\theta_\phi\|_1\le C h^{r+1}, \\  \label{projec-estimate3}
  \|\partial_t\theta_{p^i}\|+h\|\partial_t\theta_{p^i}\|_1 \le C h^{r+1}.
 \end{eqnarray}

 Finally, we introduce two lemmas which will be used in the error analysis.

\begin{lemma}\label{G-N-inequality}
 (\textbf{Gagliardo--Nirenberg inequality} \cite{L. Nirenberg 1966}) Let $u$ be a function defined on a bounded domain $\Omega \in \mathbb{R}^d$ and its derivatives of order $m$ belongs to $L^r$ in $\Omega$. Then for the derivatives $\partial^j u$, $0\le j < m$, the following inequalities hold (where constant $C$ depends only on $\Omega, m, j, q, r$)
 $$ \|\partial^j u\|_{L^p}\le C\big(\|\partial^m u\|_{L^r}^a
    \|u\|_{L^q}^{1-a}+\|u\|_{L^q} \big),$$
 for $\frac{j}{m}\le a \le 1$ with
 $$ \frac{1}{p}=\frac{j}{d}+a\Big(\frac{1}{r}-\frac{m}{d} \Big)+(1-a)\frac{1}{q}, $$
 except $1< r < \infty$ and $ m-j-\frac{d}{r}$ is a non-negative integer, in which case the above estimate holds only for $\frac{j}{m}\le a < 1$.
\end{lemma}

\begin{lemma} \cite{Y. Chen 1998} \label{poisson-regEstimate}
 Suppose that $\Omega$ is a smooth bounded domain and $u \in H^k(\Omega)$ is a solution of
 \begin{eqnarray*}
  \left\{ \begin{array} {rcl}
  -\Delta u = f, ~ x \in \Omega, \\
  \\
  u = 0, ~ x \in \partial \Omega.
  \end{array} \right.
 \end{eqnarray*}
 Then the following estimate holds for $1 < p < \infty$
 $$ \|u\|_{W^{2,p}} \le C \|f\|_{L^p}. $$
\end{lemma}

\setcounter{equation}{0}
\section{$L^2$ Norm Error Analysis for Finite Element Approximation} \label{sec3}

 \noindent
 \par In the section, we give the a priori error estimates for both the semi-discretization  finite element solution $(p_h^i,\phi_h)$ of \eqref{pwh1}-\eqref{pwh2} and the fully discrete finite element solution $(P_h^{i,n},\Phi_h^n)$ of \eqref{fph1}-\eqref{fph2}. For the sake of analysis, we assume the source term $F_3\in L^4(\ome)$ and the size of the grid $h<<1$.

\subsection{Error Analysis for the Semi-discretization} \label{subsec-semi}

\noindent
 \par We give the a priori error estimate for the semi-discretization  finite element approximation $(p_h^i,\phi_h)$ as follows.
\begin{theorem} \label{Theore-semi}
 Let $(p^i,\phi)$ and $(p_h^i,\phi_h)$  be the solutions  of \eqref{pnpw1}-\eqref{pnpw2} and \eqref{pwh1}-\eqref{pwh2}, respectively.
 Then for $t\in[0,T]$, we have the following estimate
 \begin{eqnarray} \label{semi-estimat}
  \|p^i-p_h^i\|+\|\phi-\phi_h\| \le C h^{r+1}.
 \end{eqnarray}
\end{theorem}

\begin{proof}
 From the projection error estimates \eqref{projec-estimate1}-\eqref{projec-estimate3}, we only need to estimate the following error functions
 \begin{eqnarray*}
  e_{p^i} = p_h^i-R_hp^i, ~~ e_\phi=\phi_h-R_h\phi.
 \end{eqnarray*}

 It follows from \eqref{pnpw1}-\eqref{pnpw2} and \eqref{pwh1}-\eqref{pwh2} that, $\forall v_{h}, w_h\in V_h^r$
 \begin{eqnarray} \label{semi-pr1}
  \big(\partial_t(p_h^i-p^i),v_{h}\big)
  +\big(\nabla(p_h^i-p^i),\nabla v_{h}\big)
   +q^i\big(p_h^i\nabla\phi_h-p^i\nabla\phi,\nabla v_{h}\big)=0, \\
  \big(\nabla(\phi_h-\phi),\nabla w_h\big)
       -\big(\sum\limits_{i=1}^2 q^i(p_h^i-p^i),w_h\big)=0. \label{semi-pr2}
 \end{eqnarray}
 Taking $w_h = e_\phi$ in \eqref{semi-pr2} and using \eqref{Rhprojetion-2}, we have
 \begin{align} \label{ephi-ep-en}
  \big(\nabla e_\phi,\nabla e_\phi\big)
   =\big(\sum\limits_{i=1}^2 q^ie_{p^i} ,e_\phi\big),
 \end{align}
 which easily yields
 \begin{eqnarray} \label{nablaphi-ep-en}
  \|\nabla e_\phi\| \le C \sum\limits_{i=1}^2\|e_{p^i}\|.
 \end{eqnarray}
 Taking $v_h =e_{p^i}$ in \eqref{semi-pr1} and using \eqref{Rhprojetion-1}, we have
 \begin{align} \label{semi-ep}
  (\partial_t e_{p^i}, e_{p^i})+(\nabla e_{p^i},\nabla e_{p^i})
   = \sum\limits_{j=1}^3 I_j,
 \end{align}
 where $I_j,~ j=1,2,3$, are defined as
 \begin{align*}
 I_1&:= -(\partial_t\theta_{p^i},e_{p^i}), \\
 I_2&:= -q^i\big(p^i\nabla e_\phi, \nabla e_{p^i}\big), \\
 I_3&:= -q^i\big((e_{p^i}+\theta_{p^i})\nabla\phi_h,\nabla e_{p^i}\big).
 \end{align*}

 In the following, we shall estimate $I_1,~I_2$ and $I_3$, respectively.
 By the projection estimate \eqref{projec-estimate3}, there holds
 \begin{align}  \label{semi-I1}
 I_1 \le \|\partial_t\theta_{p^i}\|\|e_{p^i}\|
 \le Ch^{r+1}\|e_{p^i}\| \le C(h^{2r+2}+\|e_{p^i}\|^2).
 \end{align}
 Using \eqref{nablaphi-ep-en} and the regularity assumption \eqref{Reg_assumption}, we have
 \begin{align}  \label{semi-I2}
  I_2 \le C\|p^i\|_{0,\infty}\|\nabla e_\phi\|\|\nabla e_{p^i}\|
      \le C \sum\limits_{i=1}^2\|e_{p^i}\|\|\nabla e_{p^i}\|
      \le C \sum\limits_{i=1}^2 \Big(\|e_{p^i}\|^2+\epsilon\|\nabla e_{p^i}\|^2\Big),
 \end{align}
 where $0<\epsilon<1$ is a constant. To estimate $I_3$, we shall prove the following result
 \begin{eqnarray} \label{phi-infty}
  \|\nabla\phi_h\|_{0,\infty}
  \le C \Big(\sum\limits_{i=1}^2\|\nabla e_{p^i}\|+ C_1\Big),
 \end{eqnarray}
 where $C_1$ is a positive constant satisfying $\sum\limits_{i=1}^2\|p^i\|_{1,2}+\|F_3\|_{0,4}\le C_1$.

 It is easy to see that $\phi_h$ can be viewed as the finite element approximation to the solution of the Poisson equation
 \begin{eqnarray}\label{phi-poisson}
  -\Delta \phi = \sum\limits_{i=1}^2 q^ip_h^i+F_3,
 \end{eqnarray}
 with homogeneous Dirichlet boundary condition.
 Hence, by $W^{1,p}$-estimate of the finite element methods \cite{S.C.Brenner 2002,R.Rannacher 1982}, Lemma \ref{G-N-inequality} and Lemma \ref{poisson-regEstimate}, we have
 \begin{align}  \nonumber
  \|\nabla\phi_h\|_{0,\infty}
  &\le C\|\phi\|_{1,\infty}\le C\|\phi\|_{2,4} \le C
       \|\sum\limits_{i=1}^2 q^ip_h^i+F_3\|_{0,4} \\  \nonumber
  &\le C \Big(\sum\limits_{i=1}^2\big(\|p_h^i-R_hp^i\|_{0,4}
        +\sum\limits_{i=1}^2\|R_hp^i\|_{0,4}+\|F_3\|_{0,4}\Big) \\
  &\le C\Big(\sum\limits_{i=1}^2\|e_{p^i}\|_{1,2}
        + \sum\limits_{i=1}^2\|p^i\|_{1,2}+\|F_3\|_{0,4}\Big)
  \le C \Big(\sum\limits_{i=1}^2\|\nabla e_{p^i}\|+ C_1\Big), \label{phi-infty-pro}
 \end{align}
 which yields estimate \eqref{phi-infty}.

 Then by \eqref{phi-infty} and the projection error estimate \eqref{projec-estimate1}, it yields
 \begin{align}  \nonumber
  I_3 &\le C\|e_{p^i}+\theta_{p^i}\|\|\nabla\phi_h\|_{0,\infty}\|
          \|\nabla e_{p^i}\| \\  \nonumber
  &\le C (\|e_{p^i}\|+ h^{r+1})
        (\sum\limits_{i=1}^2\|\nabla e_{p^i}\|
        +C_1)\|\nabla e_{p^i}\| \\   \nonumber
  &\le C\Big(\sum\limits_{i=1}^2\|e_{p^i}\|\|\nabla e_{p^i}\|^2
        +\|e_{p^i}\|\|\nabla e_{p^i}\|
        +h^{r+1}\sum\limits_{i=1}^2\|\nabla e_{p^i}\|^2
        +h^{r+1}\|\nabla e_{p^i}\| \Big) \\
 &\le C\Big(\sum\limits_{i=1}^2\|e_{p^i}\|
       \|\nabla e_{p^i}\|^2+\|e_{p^i}\|^2
       +h^{2r+2}\Big) +\epsilon\|\nabla e_{p^i}\|^2,
     \label{semi-I3}
 \end{align}
 where we have used $Ch^{r+1} \leq \epsilon$ when $h<<1$. Substituting estimates \eqref{semi-I1}-\eqref{semi-I2} and \eqref{semi-I3} into \eqref{semi-ep}, we get
 \begin{align}
  \frac{1}{2}\partial_t\|e_{p^i}\|^2+\|\nabla e_{p^i}\|^2
  &\le C\Big(\sum\limits_{i=1}^2\|e_{p^i}\|\|\nabla e_{p^i}\|^2+\|e_{p^i}\|^2
       +h^{2r+2}\Big)+\epsilon\|\nabla e_{p^i}\|^2.
      \label{pw-pro4p}
 \end{align}

 Now we conduct a mathematic induction process to prove the following inequality
 \begin{eqnarray}\label{ep-hypothesis}
  \|e_{p^i}\|\le C h^{r+1}, ~ \forall ~t\in [0,T].
 \end{eqnarray}
 Assume \eqref{ep-hypothesis} holds for any $t\in [0,T^*],~ T^*<T$. Then by \eqref{pw-pro4p},  we get
 \begin{align}
  \frac{1}{2}\partial_t\|e_{p^i}\|^2+\|\nabla e_{p^i}\|^2
  \le  C (h^{2r+2} + \|e_{p^i}\|^2)+\epsilon\|\nabla e_{p^i}\|^2.
   \label{pw-pro6}
 \end{align}
 Take integral with respect to $t$,
 \begin{eqnarray*}
  \|e_{p^i}\|^2 + \int_0^t\|\nabla e_{p^i}(s)\|^2 ds
  \le C \Big(h^{2r+2}+\int_0^t\|e_{p^i}(s)\|^2 ds \Big),
 \end{eqnarray*}
 where we have used the fact $\|e_{p^i}(0)\|=0$ by the initial condition $R_hp^i(0)=p_h^i(0)$. By using Gronwall's inequality, we have for $0\le t \le T^*$,
  \begin{eqnarray*}
  \|e_{p^i}\| \le C h^{r+1}.
 \end{eqnarray*}
 Since $\|e_{p^i}\|$ is a continuous function with respect to
 $t\in [0,T]$, due to the uniform continuity with time, then for any $\epsilon >0$, there exists $\delta$ such that for any $t\in[T^*, T^*+\delta]$,
 \begin{eqnarray*}
  \|e_{p^i}(t)-e_{p^i}(T^*)\| \le \epsilon.
 \end{eqnarray*}
 This means
 \begin{eqnarray*}
   \|e_{p^i}(t)\|\le \epsilon + Ch^{r+1} \le C h^{r+1}.
 \end{eqnarray*}
 Because $[0,T]$ is a finite interval, so the induction hypothesis \eqref{ep-hypothesis} holds true for all $t\in[0, T]$. \\

 Therefore, for any $t\in[0, T]$, by projection estimate \eqref{projec-estimate1} and \eqref{ep-hypothesis}, we can easily get
 \begin{align}
  \|p^i-p_h^i\| \le \|e_{p^i}\|+\|\theta_{p^i}\| \le C h^{r+1}. \label{semi-pph-nnh}
 \end{align}
 Then \eqref{semi-estimat} is proved by combining \eqref{nablaphi-ep-en}, the projection estimate \eqref{projec-estimate2} and \eqref{semi-pph-nnh}. This completes the proof of Theorem \ref{Theore-semi}.
\end{proof}

 Now we turn to the full discretization scheme.

\subsection{Error Analysis for the Full discretization} \label{subsec-full}

\noindent
 \par In this subsection, we present the error estimate of the full discretization schemes \eqref{fph1}-\eqref{fph2}. 
 For given $t=t^n$, define the error functions
 \begin{eqnarray} \label{full_ep_phi}
  e_{p^i}^{n} = P_h^{i,n}-R_hp^{i,n}, ~ e_\phi^{n}=\Phi_h^{n}-R_h\phi^{n}, ~ \mbox{for} ~ n=0,1,2,\cdots, N.
 \end{eqnarray}

\begin{theorem} \label{theo-full-max}
  Let $(p^{i,n},\phi^n)$ and $(P_h^{i,n},\Phi_h^n)$  be the solutions  of \eqref{pnpw1}-\eqref{pnpw2} and \eqref{fph1}-\eqref{fph2}, respectively.
 Then there exists two positive constants $\tau_0$ and $h_0$ such that for any $n=0,1,\cdots,N$,
 \begin{eqnarray} \label{max-full-estimat}
  \max\limits_{0\le n\le N}\Big(\|P_h^{i,n}-p^{i,n}\| + \|\Phi_h^n-\phi^n\|\Big) \le C (\tau+h^{r+1}),
 \end{eqnarray}
 provided by $\tau <\tau_0$ and $h \le h_0$.
\end{theorem}

\begin{proof}
 By the weak formulation \eqref{pnpw1}-\eqref{pnpw2} and the Ritz projection \eqref{Rhprojetion-1}-\eqref{Rhprojetion-2}, $\forall v_{h}, w_h \in V_h^r$, we have
 \begin{align} \nonumber
  (D_\tau p^{i,n+1},v_{h})&+(\nabla R_hp^{i,n+1},\nabla v_{h})
     +q^i(p^{i,n+1}\nabla R_h\phi^{n+1},\nabla v_{h}) \\ \label{fwpro-1}
     &= (D_\tau p^{i,n+1}, v_h)-(\partial_tp^i|_{t^{n+1}},v_{h})
        +(F_i^{n+1},v_h), \\ \label{fwpro-2}
  (\nabla R_h\phi^{n+1},\nabla w_h) &= \big(\sum\limits_{i=1}^2
     q^iR_hp^{i,n+1}, w_h \big)+(F_3^{n+1},w_h).
 \end{align}
 Then from \eqref{fwpro-1}-\eqref{fwpro-2} and the full discretization schemes \eqref{fph1}-\eqref{fph2}, we have
\begin{align} \nonumber
  \big(D_\tau &(P_h^{i,n+1}-p^{i,n+1}),v_{h}\big)
  +\big(\nabla(P_h^{i,n+1}-R_hp^{i,n+1}),\nabla v_{h}\big) \\
  & \quad= -q^i(P_h^{i,n+1}\nabla\Phi_h^{n+1}-p^{i,n+1}\nabla
          R_h\phi^{n+1},\nabla v_{h})
     -(D_\tau p^{i,n+1},v_{h})+(\partial_tp|_{t^{n+1}}, v_{h}), \label{fwpro-11} \\
  &(\nabla(\Phi_h^{n+1}-R_h\phi^{n+1}),\nabla w_h) = \sum\limits_{i=1}^2q^i\big(P_h^{i,n+1}-R_hp^{i,n+1}, w_h\big).
      \label{fwpro-21}
 \end{align}
 Choosing $v_h = e_{p^i}^{n+1}$ in \eqref{fwpro-11} and $w_h = e_\phi^{n+1}$ in \eqref{fwpro-21}, respectively, we get
 \begin{align}  \label{fwpro-f1}
  \big(D_\tau e_{p^i}^{n+1}, e_{p^i}^{n+1}\big)
   +\big(\nabla e_{p^i}^{n+1},\nabla e_{p^i}^{n+1}\big)
  = H_1^n+H_2^n+H_3^n,  \\
  (\nabla e_\phi^{n+1},\nabla e_\phi^{n+1})
  = \big(\sum\limits_{i=1}^2 q^i e_{p^i}^{n+1}, e_\phi^{n+1}\big),
      \label{fwpro-f2}
 \end{align}
 where
 \begin{align*}
  H_1^n &:= -(D_\tau\theta_{p^i}^{n+1},e_{p^i}^{n+1}), \\
  H_2^n &:= -q^i(P_h^{i,n+1}\nabla\Phi_h^{n+1}
            -p^{i,n+1}\nabla R_h\phi^{n+1},
               \nabla e_{p^i}^{n+1}), \\
  H_3^n &:= -(D_\tau p^{i,n+1}-\partial_tp^i|_{t^{n+1}},
              e_{p^i}^{n+1}).
 \end{align*}
 By \eqref{fwpro-f2}, we can easily get
 \begin{align} \label{full-nablaPHI}
 \|\nabla e_\phi^{n+1}\| \le C\sum\limits_{i=1}^2\|e_{p^i}^{n+1}\|.
 \end{align}
 Now we focus on deriving the estimates of $H_1^n,H_2^n$ and $H_3^n$.

 First by the projection estimate \eqref{projec-estimate3}, we have
 \begin{eqnarray} \nonumber
  H_1^n &\le& \frac{1}{\tau}\|\theta_{p^i}^{n+1}-\theta_{p^i}^{n}\|
        \|e_{p^i}^{n+1}\|
  \le\frac{1}{\tau}\int_{t^{n}}^{t^{n+1}}\|\partial_s(\theta_{p^i})(s)\|ds\cdot
      \|e_{p^i}^{n+1}\|  \\
  &\le& C h^{r+1}\|e_{p^i}^{n+1}\|
  \le C( h^{2r+2} +\|e_{p^i}^{n+1}\|^2), \label{fwpro-f3}
 \end{eqnarray}
 For the third term $H_3^n$, by Taylor's expansion, it yields
 \begin{align} \nonumber
  H_3^n &\le \|D_\tau p^{i,n+1}-\partial_tp^i|_{t^{n+1}}\|
            \|e_{p^i}^{n+1}\|
  \le  \|\frac{1}{2}\tau\cdot\partial_{tt}p^i(x,\xi)\|
        \|e_{p^i}^{n+1}\|  \\ \label{fwpro-f4}
  &\le C \tau \|e_{p^i}^{n+1}\| \le C (\tau^2+
         \|e_{p^i}^{n+1}\|^2), \quad (t^{n}<\xi<t^{n+1}).
 \end{align}
 For $H_2^n$, there holds
 \begin{align} \nonumber
  H_2^n
 &=-q^i\big((P_h^{i,n+1}-p^{i,n+1})\nabla\Phi_h^{n+1}
    +p^{i,n+1}\nabla(\Phi_h^{n+1}-R_h\phi^{n+1}),
     \nabla e_{p^i}^{n+1}\big) \\ \nonumber
 &= -q^i\big((e_{p^i}^{n+1}+\theta_{p^i}^{n+1})\nabla\Phi_h^{n+1},
         \nabla e_{p^i}^{n+1}\big)
    -q^i\big(p^{i,n+1}\nabla e_\phi^{n+1},\nabla e_{p^i}^{n+1}\big) \\ \nonumber
 &\le C \Big((\|e_{p^i}^{n+1}\|+h^{r+1})\|\nabla\Phi_h^{n+1}\|_{0,\infty}
      \|\nabla e_{p^i}^{n+1}\|+ \sum\limits_{i=1}^2\|e_{p^i}^{n+1}\|
      \|\nabla e_{p^i}^{n+1}\| \Big),  
 \end{align}
 where we have used \eqref{Reg_assumption}, \eqref{projec-estimate1} and \eqref{full-nablaPHI}.

 On the other hand, by inverse inequality and \eqref{full-nablaPHI},
 \begin{align} \nonumber
 \|\nabla\Phi_h^{n+1}\|_{0,\infty}
 &\le  \|\nabla R_h\phi^{n+1}\|_{0,\infty}+\|\nabla(\Phi_h^{n+1}
        -R_h\phi^{n+1})\|_{0,\infty}   \\ \nonumber
 &\le  \|\nabla\phi^{n+1}\|_{0,\infty}
       +\|\nabla(R_h\phi^{n+1}-\phi^{n+1})\|_{0,\infty}
       +\|\nabla(\Phi_h^{n+1}-R_h\phi^{n+1})\|_{0,\infty} \\ \nonumber
 &\le  \|\phi\|_{L^\infty(W^{1,\infty}(\Omega))}
      +C\big(h^r\|\phi^{n+1}\|_{r+1,\infty}
      +h^{-\frac{d}{2}}\|\nabla(\Phi_h^{n+1}-R_h\phi^{n+1})\| \big) \\
 &\le C(1 + h^{-\frac{d}{2}}\|\nabla e_\phi^{n+1}\|)
      \le C(1+ h^{-\frac{d}{2}}\sum\limits_{i=1}^2\|e_{p^i}^{n+1}\|).
      \label{full-phi-infty}
 \end{align}

 In what follows, we shall prove by mathematical induction that the following inequality holds for $n = 0,1,\cdots, N-1$
 \begin{align} \label{full-hypothesis}
  \|e_{p^i}^{n+1}\| \le C (\tau+h^{r+1}).
 \end{align}
 Assume \eqref{full-hypothesis} holds for any
 $n=0,1,\cdots,J,~ 0\le J\le N-2$. Then by \eqref{full-phi-infty}, we get $\|\nabla\Phi_h^{n+1}\|_{0,\infty} \le C $. Hence,
 \begin{align} \nonumber
  H_2^n
 &\le C \Big((\|e_{p^i}^{n+1}\|+ h^{r+1})
      \|\nabla e_{p^i}^{n+1}\|+ \sum\limits_{i=1}^2\|e_{p^i}^{n+1}\|
      \|\nabla e_{p^i}^{n+1}\|\Big) \\
 &\le C (\|e_{p^i}^{n+1}\|^2+h^{2r+2})
      +\epsilon\|\nabla e_{p^i}^{n+1}\|^2.
     \label{H2-hypro}
 \end{align}
 Combining \eqref{fwpro-f1}, \eqref{fwpro-f3}, \eqref{fwpro-f4} and \eqref{H2-hypro}, we have
 \begin{align} \label{fwpro-f6}
  \frac{1}{2}D_\tau \|e_{p^i}^{n+1}\|^2
       +\|\nabla e_{p^i}^{n+1}\|^2
  &\le C (\tau^2+h^{2r+2} +\|e_{p^i}^{n+1}\|^2)
       +\epsilon\|\nabla e_{p^i}^{n+1}\|^2.
 \end{align}
 Choosing a sufficiently small $\epsilon$ and summing up for the index $n=0,1,\cdots,J$, $0\leq J\leq N-1$ on both side of \eqref{fwpro-f6}, then we can easily get the following inequality
 \begin{align*}
  \|e_{p^i}^{J+1}\|^2+\tau\sum\limits_{n=0}^J
  \|\nabla e_{p^i}^{n+1}\|^2
  \le C (\tau^2+h^{2r+2})+C\tau\sum\limits_{n=0}^J\|e_{p^i}^{n+1}\|^2.
 \end{align*}
 By the discrete Gronwall's inequality, we get
 \begin{align}
  \|e_{p^i}^{J+1}\|^2+\tau\sum\limits_{n=0}^J
  \|\nabla e_{p^i}^{n+1}\|^2
  \le C (\tau^2+h^{2r+2}).
  \label{fwpro-f6-2}
 \end{align}
 This implies that
 \begin{eqnarray*}
 \|e_{p^i}^{J+1}\| \le C (\tau+h^{r+1}),~ 0 \le J \le N-1.
 \end{eqnarray*}
 Thus, \eqref{full-hypothesis} holds for $n=0,1,\cdots, N-1$. We complete the induction.

 Finally, by projection error estimate \eqref{projec-estimate1} and \eqref{full-hypothesis}, it yields
 \begin{eqnarray} \label{Ph-p-full}
  \max\limits_{0\le n\le N}\|P_h^{i,n}-p^{i,n}\|
  \le \max\limits_{0\le n\le N}\big(\|\theta_{p^i}^{n}\|+\|e_{p^i}^{n}\|\big)
  \le C(\tau+h^{r+1}).
 \end{eqnarray}
 Theorem \ref{theo-full-max} is proved by combining \eqref{Ph-p-full}, the projection error estimate \eqref{projec-estimate2} and \eqref{full-nablaPHI}.
\end{proof}

\begin{remark}
  Theorem \ref{theo-full-max} show that if we choose the time step $\tau$ and mesh size $h$ satisfy $\tau=O(h^{r+1})$, then the optimal $L^2$ norm error estimate is obtained when $r$-order finite element is used for both the concentration and electrostatic potential. In fact, choosing $v_h = D_\tau e_{p^i}^{n+1}$ in \eqref{fwpro-11} instead of $e_{p^i}$ and follow the analogous arguments in $H^1$ norm for the concentration and electrostatic potential in \cite{Y.Z. Sun 16}, we can prove the error estimate in $H^1$ norm, i.e.
 \begin{eqnarray} \label{full-H1-estimat}
  \|P_h^{i,n}-p^{i,n}\|_1 + \|\Phi_h^n-\phi^n\|_1 \le C (\tau+h^{r}).
 \end{eqnarray}

 Since the optimal $H^1$ norm error estimate is proved in \cite{Y.Z. Sun 16} and the similar proof fashion shall be presented in Section \ref{sec4}, for the sake of simplicity, the detained proof of \eqref{full-H1-estimat} is omitted here.
\end{remark}

 Next, a two-grid finite element method for PNP equation \eqref{pnpt1} will be presented in full discretization schemes. Some error estimates are derived  which show our method can achieve the same error accuracy as the standard finite element method. However, a much less CPU time cost which is shown by the numerical experiments in Section \ref{sec-num-res}.

 \setcounter{equation}{0}
 \section{The Two-Grid Algorithm and Error Analysis} \label{sec4}

\noindent
 \par In this section, we shall present the main algorithms of the paper. Two quasi-uniform triangulations $\mathcal{T}_H$ and $\mathcal{T}_h$ of $\Omega$ with two different  mesh sizes $H$ and $h \ (H > h)$ are introduced. The corresponding finite element spaces $V_H^r$ and $V_h^r$, which satisfy $V_H^r\subset V_h^r$ are called the coarse-grid and fine-grid space, respectively. Two algorithms are provided to decouple the strong coupled equations and some error estimates are also derived.

 First, a semi-decoupled scheme is presented as follows:
\begin{algo}\label{al-FI}(Semi-decoupled scheme)\\
 Step 1. Given $(P_H^{i,n},\Phi_H^n)\in V_H^r$, $i=1,2$, find $(P_H^{i,n+1}, \Phi_H^{n+1})\in V_H^r$,  such that
 \begin{align} \label{aIF11}
 (D_\tau P_H^{i,n+1},v_H)+(\nabla P_H^{i,n+1},\nabla v_H)
 +(q^iP_H^{i,n+1}\nabla\Phi_H^{n+1},\nabla v_H)
 &= (F_i^{n+1},v_H),\quad \forall v_H\in V_H^r, \\ \label{aIF12}
 (\nabla\Phi_H^{n+1},\nabla w_H)-\sum\limits_{i = 1}^2 q^i(P_H^{i,n+1},w_H)
 &= (F_3^{n+1},w_H),\quad \forall w_H\in V_H^r.
 \end{align}
 Step 2. Given  $P_H^{i,n+1} \in V_H^r$ and $(P_{h*}^{i,n},\Phi_{h*}^n)\in V_h^r$, find $(P_{h*}^{i,n+1},\Phi_{h*}^{n+1})\in V_h^r$, such that
 \begin{align} \label{aIF21}
 (D_\tau P_{h*}^{i,n+1},v_h)+(\nabla P_{h*}^{i,n+1},\nabla v_h)+(q^iP_{h*}^{i,n+1}\nabla\Phi_{h*}^{n+1},\nabla v_h)
 &= (F_i^{n+1},v_h),\quad \forall v_h\in V_h^r, \\
  \label{aIF22}
 (\nabla\Phi_{h*}^{n+1},\nabla w_h)-\sum\limits_{i = 1}^2 q^i(P_H^{i,n+1},w_h)
 &= (F_3^{n+1},w_h),\quad \forall w_h\in V_h^r,
 \end{align}
 where
 $$D_\tau P_H^{i,n+1}=\frac{P_H^{i,n+1}-P_H^{i,n}} { \tau} , ~\mbox{for}~ n=0,1,2,\cdots,N-1,$$
 and
  $$D_\tau P_{h*}^{i,n+1}=\frac{P_{h*}^{i,n+1}-P_{h*}^{i,n}} { \tau} , ~\mbox{for}~ n=0,1,2,\cdots,N-1.$$
  At the initial step, we take $p_{h*}^{i,0} = R_hp^{i,0}$, where $R_h$ is the Ritz projection operator defined in \eqref{Rhprojetion-1}-\eqref{Rhprojetion-2}.
 \end{algo}

 We need the following error estimate in the later analysis.
\begin{lemma} \label{lemma-tph-p}
  Let $(p^{i,n+1},\phi^{n+1})$ and $(P_{h*}^{i,n+1},\Phi_{h*}^{n+1})$  be the solutions  of \eqref{pnpw1}-\eqref{pnpw2} and \eqref{aIF21}-\eqref{aIF22}, respectively. Then for any $n=0,1,\cdots,N-1$, we have
 \begin{align} \label{TG-full-F01}
  \|P_{h*}^{i,n+1}-p^{i,n+1}\|  &\le C (\tau+H^{r+1}).
 \end{align}
\end{lemma}
\begin{proof}
 Denote
 \begin{eqnarray}\label{tgph-p_phih-phi}
  P_{h*}^{i,n+1}-p^{i,n+1} = \rho_{p^i}^{n+1}+\theta_{p^i}^{n+1}, ~ \Phi_{h*}^{n+1}-\phi^{n+1}= \rho_{\phi}^{n+1}+\theta_{\phi}^{n+1},
 \end{eqnarray}
 where
 \begin{align*}
 &\rho_{p^i}^{n+1}=P_{h*}^{i,n+1}-R_hp^{i,n+1}, ~ \theta_{p^i}^{n+1}=R_hp^{i,n+1}-p^{i,n+1}, \\
 &\rho_{\phi}^{n+1} = \Phi_{h*}^{n+1}-R_h\phi^{n+1}, ~ \quad\theta_\phi^{n+1}=R_h\phi^{n+1}-\phi^{n+1}.
 \end{align*}

 Following a similar proof of Theorem \ref{theo-full-max}, subtracting \eqref{fwpro-1} from \eqref{aIF21}, and taking $v_h=\rho_{p^i}^{n+1}$, we have the error equation
 \begin{align} \label{TG-3}
  \big(D_\tau\rho_{p^i}^{n+1},\rho_{p^i}^{n+1}\big)
  +\big(\nabla\rho_{p^i}^{n+1},\nabla \rho_{p^i}^{n+1}\big)= T_1^n+T_2^n+T_3^n,
 \end{align}
 where
 \begin{align*}
 & T_1^n = -\big(D_\tau\theta_{p^i}^{n+1},\rho_{p^i}^{n+1}\big), \\
 & T_2^n = -(D_\tau p^{i,n+1}-\partial_tp|_{t^{n+1}}, \rho_{p^i}^{n+1}), \\
 & T_3^n = -q^i\big((\rho_{p^i}^{n+1}+\theta_{p^i}^{n+1})\nabla\Phi_{h*}^{n+1},
         \nabla \rho_{p^i}^{n+1}\big) -q^i\big(p^{i,n+1}\nabla\rho_\phi^{n+1},\nabla\rho_{p^i}^{n+1}\big).
 \end{align*}
 We shall estimate $T_1^n, T_2^n$ and $T_3^n$, respectively below.

 By the similar arguments as in \eqref{fwpro-f3}-\eqref{fwpro-f4}, we get
 \begin{align} \label{tg711}
  T_1^n &\le C(h^{2r+2}+\|\rho_{p^i}^{n+1}\|^2), \\
  T_2^n &\le C(\tau^2+\|\rho_{p^i}^{n+1}\|^2).
   \label{tg712}
 \end{align}
 To estimate the third term, $T_3^n$, we need the fact $\|\nabla\Phi_{h*}^{n+1}\|_{0,\infty}\le C$,  and the estimate of $\|\nabla\rho_\phi^{n+1}\|$.

 In fact, by \eqref{fwpro-2} and \eqref{aIF22}, $\forall w_h \in V_h^r$, we have
 \begin{align}\label{TG-4}
  (\nabla\rho_{\phi}^{n+1},\nabla w_h)
  = \sum\limits_{i=1}^2q^i\big(P_H^{i,n+1}-R_hp^{i,n+1}, w_h\big).
 \end{align}
 Taking $w_h = \rho_{\phi}^{n+1}$ in \eqref{TG-4}, we can easily get
 \begin{align}\label{TG-ff1}
  \|\nabla\rho_{\phi}^{n+1}\|
  \le \sum\limits_{i=1}^2\|P_H^{i,n+1}-R_hp^{i,n+1}\|
  \le C(\tau+H^{r+1}).
 \end{align}
 Then $\|\nabla\Phi_{h*}^{n+1}\|_{0,\infty}\le C$ holds by using \eqref{TG-ff1} and the same arguments as in \eqref{full-phi-infty}. By the regularity assumption \eqref{Reg_assumption}, the projection estimate \eqref{projec-estimate1} and \eqref{TG-ff1}, $T_3^n$ is estimated by
 \begin{align} \nonumber
  T_3^n
  &\le C\big(\|\rho_{p^i}^{n+1}+\theta_{p^i}^{n+1}\|\|\nabla\Phi_{h*}^{n+1}\|_{0,\infty}
     \|\nabla\rho_{p^i}^{n+1}\| + \|p^{i,n+1}\|_{0,\infty}
     \|\nabla\rho_\phi^{n+1}\|\|\nabla\rho_{p^i}^{n+1}\| \big) \\ \nonumber
 &\le C \big((\|\rho_{p^i}^{n+1}\|+h^{r+1})\|\nabla\rho_{p^i}^{n+1}\|+
        (\tau+H^{r+1})\|\nabla\rho_{p^i}^{n+1}\| \big) \\
 &\le C \big(\tau^2+H^{2r+2}+\|\rho_{p^i}^{n+1}\|^2 \big)
        +\epsilon\|\nabla \rho_{p^i}^{n+1}\|^2.
      \label{TG-5}
 \end{align}
 Thus, by \eqref{tg711}, \eqref{tg712} and \eqref{TG-5}, equation \eqref{TG-3} becomes
 \begin{eqnarray}  \label{TG-6}
  \frac{1}{2}D_\tau\|\rho_{p^i}^{n+1}\|^2+\|\nabla\rho_{p^i}^{n+1}\|^2
  \le C(\tau^2+H^{2r+2}+\|\rho_{p^i}^{n+1}\|^2) + \epsilon\|\nabla\rho_{p^i}^{n+1}\|^2.
 \end{eqnarray}
 Applying a summation of time step $n$ from $0$ to $J$ on both side of \eqref{TG-6}, where $0\le J\le N-1$, we get the following inequality
 \begin{eqnarray*}
  \frac{1}{2\tau}\|\rho_{p^i}^{J+1}\|^2+\sum\limits_{n=0}^J \|\nabla\rho_{p^i}^{n+1}\|^2
  \le C\sum\limits_{n=0}^J\big(\tau^2+H^{2r+2}+\|\rho_{p^i}^{n+1}\|^2+ \epsilon\|\nabla\rho_{p^i}^{n+1}\big).
 \end{eqnarray*}
 Then by discrete Gronwall's inequality, it yields
 \begin{eqnarray*}
  \|\rho_{p^i}^{J+1}\|^2+\tau\sum\limits_{n=0}^J \|\nabla\rho_{p^i}^{n+1}\|^2
  \le C(\tau^2+H^{2r+2}).
 \end{eqnarray*}
 This implies that for $0\le J\le N-1$,
 $$\|\rho_{p^i}^{J+1}\| \le C(\tau+H^{r+1}).$$

 Finally, by triangle inequality and projection estimate \eqref{projec-estimate1}, for $n=0,1,\cdots,N-1$, we can easily get
 \begin{eqnarray*}
  \|P_{h*}^{i,n+1}-p^{i,n+1}\| \le \|\rho_{p^i}^{n+1}\|+\|\theta_{p^i}^{n+1}\|
   \le C(\tau+H^{r+1}).
 \end{eqnarray*}
 This completes the proof.
\end{proof}

\begin{theo} \label{theo-F1}
 Suppose $(p^{i,{n+1}},\phi^{n+1})$ and $(P_{h*}^{i,n+1},\Phi_{h*}^{n+1})$ are the solutions of \eqref{pnpw1}-\eqref{pnpw2} and \eqref{aIF21}-\eqref{aIF22}, respectively.
 Then for any $n=0,1,\cdots,N-1$, we have the following estimate
 \begin{eqnarray} \label{TG-H1-error}
  \|\Phi_{h*}^{n+1}-\phi^{n+1}\|_1 + \|P_{h*}^{i,n+1}-p^{i,n+1}\|_1
  \le C (\tau+ h^{r}+H^{r+1}).
 \end{eqnarray}
\end{theo}

\begin{proof}
 First by \eqref{tgph-p_phih-phi}, \eqref{TG-ff1} and the projection estimate \eqref{projec-estimate2}, it easily yields
 \begin{eqnarray} \label{TG-H1-error0722}
  \|\Phi_{h*}^{n+1}-\phi^{n+1}\|_1 \le C (\tau+ h^{r}+H^{r+1}).
 \end{eqnarray}
 Now we turn to estimate $\|P_{h*}^{i,n+1}-p^{i,n+1}\|_1$.

 Choosing $v_h = D_\tau\rho_{p^i}^{n+1}$ in \eqref{TG-3} instead of $\rho_{p^i}^{n+1}$, we get
 \begin{align} \label{TG-h11}
  \big(\nabla\rho_{p^i}^{n+1},D_\tau\nabla\rho_{p^i}^{n+1}\big)
  +\big(D_\tau\rho_{p^i}^{n+1},D_\tau\rho_{p^i}^{n+1}\big)
  = \sum\limits_{i=1}^3\hat H_i^n,
 \end{align}
 where
 \begin{align*}
 \hat H_1^n &= -\big(D_\tau\theta_{p^i}^{n+1},D_\tau\rho_{p^i}^{n+1}\big)
             \le C h^{2r+2}+\epsilon\|D_\tau\rho_{p^i}^{n+1}\|^2,  \\
 \hat H_2^n &= -(D_\tau p^{i,n+1}-\partial_tp|_{t^{n+1}}, D_\tau\rho_{p^i}^{n+1})
             \le C \tau^2+\epsilon\|D_\tau\rho_{p^i}^{n+1}\|^2, \\
 \hat H_3^n &= -q^i\big((P_{h*}^{i,n+1}-p^{i,n+1})\nabla\Phi_{h*}^{n+1}
               +p^{i,n+1}\nabla(\Phi_{h*}^{n+1}-R_h\phi^{n+1}),
               \nabla D_\tau\rho_{p^i}^{n+1}\big) \\
   &=-q^i\frac{1}{\tau}\big[\big((\rho_{p^i}^{n+1}
               +\theta_{p^i}^{n+1})\nabla\Phi_{h*}^{n+1},
                 \nabla(\rho_{p^i}^{n+1}-\rho_{p^i}^{n})\big)
               +\big(p^{i,n+1}\nabla\rho_\phi^{n+1},\nabla(\rho_{p^i}^{n+1}
               -\rho_{p^i}^{n})\big)\big].
 \end{align*}
 On the other hand, by \eqref{TG-full-F01}, \eqref{TG-ff1} and $\|\nabla\Phi_{h*}^{n+1}\|_{0,\infty}\le C$, there holds
 \begin{align*}
  \big((\rho_{p^i}^{n+1}+\theta_{p^i}^{n+1})\nabla\Phi_{h*}^{n+1},
    \nabla(\rho_{p^i}^{n+1}-\rho_{p^i}^{n})\big)
    &=\big((\rho_{p^i}^{n+1}+\theta_{p^i}^{n+1})\nabla\Phi_{h*}^{n+1},
            \nabla\rho_{p^i}^{n+1}\big)
       -\big((\rho_{p^i}^{n}+\theta_{p^i}^{n})\nabla\Phi_{h*}^{n},
            \nabla\rho_{p^i}^{n}\big) \\
    &\quad -\big((\rho_{p^i}^{n+1}+\theta_{p^i}^{n+1})\nabla\Phi_{h*}^{n+1}
           -(\rho_{p^i}^{n}+\theta_{p^i}^{n})\nabla\Phi_{h*}^{n},
              \nabla\rho_{p^i}^{n}\big) \\
    &\le \big((\rho_{p^i}^{n+1}+\theta_{p^i}^{n+1})\nabla\Phi_{h*}^{n+1},
            \nabla\rho_{p^i}^{n+1}\big)
           -\big((\rho_{p^i}^{n}+\theta_{p^i}^{n})\nabla\Phi_{h*}^{n},
            \nabla\rho_{p^i}^{n}\big) \\
    &\quad   +C(\tau^2+H^{2r+2}+\|\nabla\rho_{p^i}^{n}\|^2),
  \end{align*}
 and
 \begin{align*}
  \big(p^{i,n+1}\nabla\rho_\phi^{n+1},\nabla\rho_{p^i}^{n+1}
               -\rho_{p^i}^{n})\big)
    &= \big(p^{i,n+1}\nabla\rho_\phi^{n+1},\nabla\rho_{p^i}^{n+1}\big)
       -\big(p^{i,n}\nabla\rho_\phi^{n},\nabla\rho_{p^i}^{n}\big) \\
    &\quad -\big(p^{i,n+1}\nabla\rho_\phi^{n+1}-p^{i,n}\nabla\rho_\phi^{n},
                 \nabla\rho_{p^i}^{n}\big) \\
    &\le \big(p^{i,n+1}\nabla\rho_\phi^{n+1},\nabla\rho_{p^i}^{n+1}\big)
        -\big(p^{i,n}\nabla\rho_\phi^{n},\nabla\rho_{p^i}^{n}\big) \\
    &\quad +C(\tau^2+H^{2r+2}+\|\nabla\rho_{p^i}^{n}\|^2).
 \end{align*}
 Then the third term estimated by
 \begin{align*}
  \hat H_3^n
  &\le C\frac{1}{\tau}\big[\big((\rho_{p^i}^{n+1}
       +\theta_{p^i}^{n+1})\nabla\Phi_{h*}^{n+1},\nabla\rho_{p^i}^{n+1}\big)
       -\big((\rho_{p^i}^{n}+\theta_{p^i}^{n})\nabla\Phi_{h*}^{n},
               \nabla\rho_{p^i}^{n}\big) \\
  &\quad +\big(p^{i,n+1}\nabla\rho_\phi^{n+1},\nabla\rho_{p^i}^{n+1}\big)
       -\big(p^{i,n}\nabla\rho_\phi^{n},\nabla\rho_{p^i}^{n}\big)
       +\tau^2+H^{2r+2}+\|\nabla\rho_{p^i}^{n}\|^2\big].
 \end{align*}
 Inserting the error estimates of $\hat H_1^n, \hat H_2^n$ and $\hat H_3^n$ into \eqref{TG-h11}, it yields
 \begin{align} \nonumber
  \frac{1}{2}D_\tau\|\nabla\rho_{p^i}^{n+1}\|^2 &+\|D_\tau\rho_{p^i}^{n+1}\|^2
  \le C (\tau^2+h^{2r+2}+\|\nabla\rho_{p^i}^{n}\|^2)
    +\epsilon\|D_\tau\rho_{p^i}^{n+1}\|^2 \\ \nonumber
  &\quad +C\frac{1}{\tau}\big[\big((\rho_{p^i}^{n+1}
       +\theta_{p^i}^{n+1})\nabla\Phi_{h*}^{n+1},\nabla\rho_{p^i}^{n+1}\big)
       -\big((\rho_{p^i}^{n}+\theta_{p^i}^{n})\nabla\Phi_{h*}^{n},
               \nabla\rho_{p^i}^{n}\big) \\
  &\quad +\big(p^{i,n+1}\nabla\rho_\phi^{n+1},\nabla\rho_{p^i}^{n+1}\big)
       -\big(p^{i,n}\nabla\rho_\phi^{n},\nabla\rho_{p^i}^{n}\big)
       +\tau^2+H^{2r+2}+\|\nabla\rho_{p^i}^{n}\|^2\big].
   \label{TG-h12}
 \end{align}
 Multiplying the time step size $\tau$ on both side of \eqref{TG-h12},
 and applying a summation of time step $n$ from $0$ to $J$,
 where $0\leq J\leq N-1$, by using \eqref{TG-full-F01}, we get
 \begin{align*} \nonumber
  \|\nabla\rho_{p^i}^{J+1}\|^2
  & +\tau\sum\limits_{n=0}^J\|D_\tau\rho_{p^i}^{n+1}\|^2
    \le C\Big(\sum\limits_{n=0}^J \big(\tau^2+H^{2r+2}+\|\nabla\rho_{p^i}^n\|^2
      + \epsilon\|D_\tau\rho_{p^i}^{n+1}\|^2 \big) \\  \nonumber
  & +\big((\rho_{p^i}^{J+1}+\theta_{p^i}^{J+1})\nabla\Phi_{h*}^{J+1},
            \nabla\rho_{p^i}^{J+1}\big)
    +\big(p^{i,J+1}\nabla\rho_\phi^{J+1},\nabla\rho_{p^i}^{J+1}\big) \Big)
        \\ \nonumber
  &\le C\sum\limits_{n=0}^J \big(\tau^2+H^{2r+2}+\|\nabla\rho_{p^i}^n\|^2
      + \epsilon\|D_\tau\rho_{p^i}^{n+1}\|^2 \big)
      +C(\tau^2+H^{2r+2}+\|\nabla\rho_{p^i}^{J+1}\|^2)  \\
  &\le C\big(\tau^2+H^{2r+2}+\sum\limits_{n=0}^J\|\nabla\rho_{p^i}^{n+1}\|^2\big)
       +\epsilon\|D_\tau\rho_{p^i}^{n+1}\|^2,
 \end{align*}
  where we have used $\rho_{p^i}^0 =0$ by the initial condition $p_{h*}^{i,0} = R_hp^{i,0}$.

 Applying the discrete Gronwall's inequality, it easily yields
 \begin{eqnarray}  \label{tp-F2-g9}
  \|\nabla\rho_{p^i}^{J+1}\| \le C(\tau+H^{r+1}).
 \end{eqnarray}
  Thus, by \eqref{tp-F2-g9} and the projection estimate \eqref{projec-estimate1}, for any $n=0,1,\cdots,N-1$, we can easily get
  \begin{eqnarray} \label{th-JN}
   \|P_{h*}^{i,n+1}-p^{i,n+1}\|_1 \le C(\tau+h^r+H^{r+1}).
  \end{eqnarray}
 Then the desired result \eqref{TG-H1-error} is
 completed by \eqref{TG-H1-error0722} and \eqref{th-JN}.
\end{proof}

\begin{remark}
  Theorem \ref{theo-F1} show that if we choose the mesh size
  $h^r=\mathcal{O}(H^{r+1})$ for $r$-th finite element, then the two-grid method can reach the same convergence order as the standard finite element method for both the electrostatic potential and concentration in $H^1$ norm. For example, if we choose the linear finite element to discrete PNP equation, i.e. $r =1$ in this case, then our two-grid method can achieve the same convergence rate when $H = O(h^{\frac 1 2})$.
 %

\end{remark}

 In the following we give another two-grid algorithm which is called the full decoupled scheme. Since the arguments of the error analysis are similar to the semi-decoupled Algorithm \ref{theo-F1}, the detailed analysis shall not be presented here again.

 \begin{algo}\label{aII-FII}(Full decoupled scheme)\\
 Step 1. Given $(P_H^{i,n},\Phi_H^n)\in V_H^r, \ i=1,2$, find $(P_H^{i,n+1}, \Phi_H^{n+1})\in V_H^r$, such that
 \begin{eqnarray} \label{aII-F11}
 (D_\tau P_H^{i,n+1},v_H)+(\nabla P_H^{i,n+1},\nabla v_H)
 +(q^iP_H^{i,n+1}\nabla\Phi_H^{n+1},\nabla v_H)
 = (F_i^{n+1},v_H),\quad \forall v_H\in V_H^r, \\ \label{aII-F12}
 (\nabla\Phi_H^{n+1},\nabla w_H)-\sum\limits_{i = 1}^2 q^i(P_H^{i,n+1},w_H)
 = (F_3^{n+1},w_H),\quad \forall w_H\in V_H^r.
 \end{eqnarray}
 Step 2. Given $(\Phi_H^{n+1}, P_H^{i,n+1}) \in V_H^r$ and $(P_{h*}^{i,n},\Phi_{h*}^n)\in V_h^r$, find $(P_{h*}^{i,n+1},\Phi_{h*}^{n+1})\in V_h^r$, such that
 \begin{eqnarray} \label{aII-F21}
 (D_\tau P_{h*}^{i,n+1},v_h)+(\nabla P_{h*}^{i,n+1},\nabla v_h)
 +(q^iP_{h*}^{i,n+1}\nabla\Phi_H^{n+1},\nabla v_h)
 = (F_i^{n+1},v_h),\quad \forall v_h\in V_h^r, \\ \label{aII-F22}
 (\nabla\Phi_{h*}^{n+1},\nabla w_h)-\sum\limits_{i = 1}^2 q^iP_H^{i,n+1},w_h)
 = (F_3^{n+1},w_h),\quad \forall w_h\in V_h^r,
 \end{eqnarray}
 where
 $$D_\tau P_H^{i,n+1}=\frac{P_H^{i,n+1}-P_H^{i,n}} { \tau} , ~\mbox{for}~ n=0,1,2,\cdots,N-1,$$
 and
 $$D_\tau P_{h*}^{i,n+1}=\frac{P_{h*}^{i,n+1}-P_{h*}^{i,n}}{\tau}, ~\mbox{for}~ n=0,1,2,\cdots,N-1.$$
 \end{algo}

 Compared with Algorithm \ref{al-FI}, the finite element approximation $\Phi_H^{n+1}$ on the coarse grid is also used to decouple the system on the fine grid in Algorithm \ref{aII-FII}. Since the system \eqref{aII-F21}-\eqref{aII-F22} is fully decoupled, it can be solved in parallel on the fine grid level.

  Similar to Theorem \ref{theo-F1}, we have the following result.

\begin{theo} \label{theo-ful1decoupled}
 Let $(p^{i,{n+1}},\phi^{n+1})$ be the solution of \eqref{pnpw1}-\eqref{pnpw2}. Assume $(P_{h*}^{i,n+1},\Phi_{h*}^{n+1})$ is the solution obtained by Algorithm \ref{aII-FII}. Then for any $n=0,1,\cdots,N-1$, we have the following estimate
 \begin{eqnarray} \label{TG-fulldecoupled-H1-error}
  \|\Phi_{h*}^{n+1}-\phi^{n+1}\|_1 + \|P_{h*}^{i,n+1}-p^{i,n+1}\|_1
  \le C (\tau+ h^{r}+H^{r+1}).
 \end{eqnarray}
\end{theo}

\begin{remark}
  Theorem \ref{theo-ful1decoupled} shows that the optimal convergence rate for both the electrostatic potential and concentration in $H^1$ norm could be reached when $h^r=\mathcal{O}(H^{r+1})$, which indicates our two-grid method retains the same order of accuracy as the standard finite element method under the assumption $H \le C h^{\frac{r}{r+1}}$.
  Moreover, since Algorithm \ref{aII-FII} is full decoupled in step 2, it can be solved in parallel on the fine grid level at each time step, the efficiency of which could be much better than the standard finite element method.
\end{remark}

\setcounter{equation}{0}
\section{Numerical Experiments} \label{sec-num-res}

\noindent
\par We now present numerical experiments to demonstrate the effectiveness and efficiency of the two-grid approach. To implement the algorithms, for the first example, the code is written in Fortran 90 and all the computations are carried out on the computer with Dual core 96 GB RAM HPZ280. The second one is carried out by Matlab R2012a on a microcomputer and the programme is under the frame work of iFEM toolbox (https://bitbucket.org/ifem/ifem).


\begin{figure}[H]
 \centerline{
 \includegraphics[height=5.0cm, width=5.0cm]{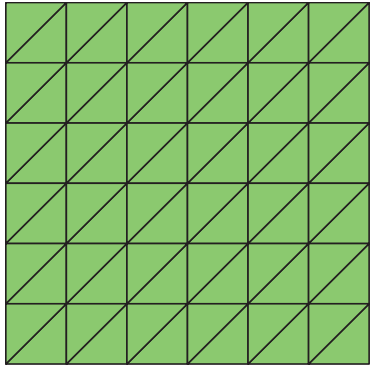} }
 \caption{ A uniform triangulation on the unit square with $M = 6$.}
 \label{uniform_mesh}
\end{figure}

\noindent
\emph{Example 5.1}
 Let the computational domain be the unit square $\Omega = [0,1]\times [0,1]$, and a uniform triangular partition with $M +1$ nodes in each direction is used. An illustration with $M = 6$ is shown in Fig. \ref{uniform_mesh}. For the coarse grid space and the fine grid space,
 the domain $\Omega$ is uniformly divided by the triangulation of mesh size $H$ and $h$, respectively.

 We choose $q^1=1, q^2=-1$ and consider the following PNP equations (cf. \cite{Y.Z. Sun 16})
 \begin{eqnarray}  \label{pnp_num}
 \left\{ \begin{array} {rcl}
   {\frac{\partial p^1}{\partial t} } - \nabla\cdot(\nabla {p^1} + {p^1}\nabla \phi ) = {F_1},   \\
   \\
   {\frac{\partial p^2}{\partial t} } - \nabla\cdot(\nabla {p^2} - {p^2}\nabla \phi ) = {F_2},   \\
   \\
   -\Delta \phi  - (p^1-p^2) = {F_3}.
  \end{array} \right.
 \end{eqnarray}
 The initial-boundary condition and the right-hand side functions $F_i,~i=1,2,3$ are chosen such that the exact solutions of \eqref{pnp_num} are given by
 \begin{eqnarray*}
 \left\{ \begin{array} {rcl}
   p^1(t,x,y) &=& \sin(t)\sin(2\pi x)\sin(2\pi y),  \\
   p^2(t,x,y) &=& \sin(t)\sin(3\pi x)\sin(3\pi y),   \\
   \phi(t,x,y) &=& (1-e^{-t})\sin(\pi x)\sin(\pi y).
  \end{array} \right.
 \end{eqnarray*}

 In the following, we first present the numerical results of standard finite element method \eqref{fph1}-\eqref{fph2},  and then show the results of Algorithm \ref{al-FI} and \ref{aII-FII}.

 To solve the nonlinear coupled system \eqref{fph1}-\eqref{fph2}, we use the following algorithm  which is introduced in \cite{Y.Z. Sun 16, A. Prohl 09} to get the finite element solution.

 \begin{algo} \label{ALGO_num}
 Step 1. Initialization for the time marching: Set time step $n=0$, and get the initial value $(P_{h}^{1,0},P_{h}^{2,0},\Phi_h^0) \in [V_h^r]^3$.
 \\
 Step 2. Initialization for nonlinear iteration: Let $(P_{h}^{1,n+1,0},P_{h}^{2,n+1,0},\Phi_h^{n+1,0}) = (P_{h}^{1,n},P_{h}^{2,n},\Phi_h^{n}) $ when $n \ge 0$ and $l = 0$.
 \\
 Step 3. Finite element computation on each time level: For $l \ge 0$, compute \\ $(P_{h}^{1,n+1,l+1},P_{h}^{2,n+1,l+1},\Phi_h^{n+1,l+1}) \in [V_h^r]^3$, such that for all $(v_{1h},v_{2h},w_h) \in [V_h^r]^3$,
 \begin{align*}
  \frac{1}{\tau}(P_{h}^{1,n+1,l+1}, v_{1h})+(\nabla P_{h}^{1,n+1,l+1}, \nabla v_{1h}) + (P_{h}^{1,n+1,l+1}\nabla\Phi_h^{n+1,l}, \nabla v_{1h} )
  &= (F_1^{n+1}, v_{1h}) + \frac{1}{\tau}(P_{h}^{1,n}, v_{1h}), \\
  \\
  \frac{1}{\tau}(P_{h}^{2,n+1,l+1}, v_{2h})+(\nabla P_{h}^{2,n+1,l+1}, \nabla v_{1h}) + (P_{h}^{2,n+1,l+1}\nabla\Phi_h^{n+1,l}, \nabla v_{2h} )
  &= (F_2^{n+1}, v_{2h}) + \frac{1}{\tau}(P_{h}^{2,n}, v_{2h}), \\
  \\
  (\nabla \Phi_h^{n+1,l+1},\nabla w_h)-( P_{h}^{1,n+1,l+1}-P_{h}^{2,n+1,l+1},w_h)
  &= (F_3^{n+1}, w_h).
 \end{align*}
 Step 4. Checking the stopping criteria for nonlinear iteration: For a given  tolerance $\epsilon$, stop the iteration when
  \begin{eqnarray*}
  \|P_{h}^{1,n+1,l+1}-P_{h}^{1,n+1,l}\|+\|P_{h}^{2,n+1,l+1}-P_{h}^{2,n+1,l}\| +\|\Phi_h^{n+1,l+1}-\Phi_h^{n+1,l}\| \le \epsilon,
  \end{eqnarray*}
 and set $(P_{h}^{1,n+1}, P_{h}^{2,n+1}, \Phi_h^{n+1})
 =(P_{h}^{1,n+1,l+1},P_{h}^{2,n+1,l+1}, \Phi_h^{n+1,l+1})$. Otherwise, set $l\leftarrow l+1$ and go to Step 3 to continue the nonlinear iteration.
 \\
 Step 5. Time marching: Stop if $n+1 = N$. Otherwise, set $n \leftarrow n+1$, and go to Step 2.
 \end{algo}

 In our computation, the piecewise linear finite elements on a uniform triangular mesh are used to discretize the PNP equations. The Gummel iteration \eqref{gummel} is used during the finite element computation on each time level in Step 3. We choose the time step $\tau = h^2$ and set the final time $T=0.5$. The tolerance $\epsilon =1.0 \times 10^{-6}$  is chosen for the nonlinear iteration in Algorithm \ref{ALGO_num}. Particularly, we adopt the AMG-PCG and AMG-PGMRES solver to solve the algebraic system ``$Ax=b$" for the Poisson equation and Nernst-Planck equations, respectively, and the inneriteration stopped if the Euclidean norm of the residual vector is less than $10^{-8}$. The numerical results in Table \ref{ch4ex1:stdFEM_jishuL2_t=0.5} and Table \ref{ch4ex1:stdFEM_t=0.5} show that the errors for $\Phi_h$ and $P_h^i \ (i=1,2)$ in $L^2$ norm and $H^1$ norm are second-order and first-order reduction, respectively, which coincides with the convergence theory shown in Theorem \ref{theo-full-max} and \eqref{full-H1-estimat}.
\begin{table}[H]
\centering
\caption{ $L^2$ error of the standard finite element method  }
\begin{tabular}{p{1.52cm}p{2.5cm}p{1.52cm}p{2.5cm}p{1.52cm}p{2.5cm}p{.8cm}}
 \toprule 
 $h$ &$\|\Phi_h-\phi\|$ &Order& $\|P_h^{1}-p^1\|$ & Order & $\|P_h^{2}-p^2\|$ & Order  \\
 \midrule 
 1/9  & 7.3983E-03 & $-$  & 3.2614E-02 & $-$  & 1.2117E-01 & $-$   \\
 1/16 & 2.4124E-03 & 1.95 & 1.0904E-02 & 1.90 & 4.2949E-02 & 1.80   \\
 1/25 & 9.9267E-04 & 1.99 & 4.5135E-03 & 1.98 & 1.8098E-02 & 1.94   \\
 1/36 & 4.8039E-04 & 1.99 & 2.1894E-03 & 1.98 & 8.8305E-03 & 1.97    \\
 1/49 & 2.5946E-04 & 2.00 & 1.1835E-03 & 2.00 & 4.7870E-03 & 1.99   \\
 1/64 & 1.5221E-04 & 2.00 & 6.9466E-04 & 2.00 & 2.8131E-03 & 1.99    \\
 \bottomrule 
\end{tabular}  \label{ch4ex1:stdFEM_jishuL2_t=0.5}
\end{table}

\begin{table}[H]
\centering
\caption{ $H^1$ error of the standard finite element method  }
\begin{tabular}{p{1.52cm}p{2.5cm}p{1.52cm}p{2.5cm}p{1.52cm}p{2.5cm}p{.8cm}}
 \toprule 
 $h$ &$\|\Phi_h-\phi\|_1$ &Order& $\|P_h^{1}-p^1\|_1$ & Order & $\|P_h^{2}-p^2\|_1$ & Order  \\
 \midrule 
 1/9  & 1.5014E-01 & $-$  & 7.1128E-01 & $-$  & 2.6894E+00 & $-$   \\
 1/16 & 8.5653E-02 & 0.98 & 4.1627E-01 & 0.93 & 1.6096E+00 & 0.89  \\
 1/25 & 5.4812E-02 & 1.00 & 2.7032E-01 & 0.97 & 1.0454E+00 & 0.97  \\
 1/36 & 3.8128E-02 & 1.00 & 1.9159E-01 & 0.94 & 7.3134E-01 & 0.98  \\
 1/49 & 2.8011E-02 & 1.00 & 1.4457E-01 & 0.91 & 5.3916E-01 & 0.99   \\
 1/64 & 2.1458E-02 & 1.00 & 1.1497E-01 & 0.86 & 4.1421E-01 & 0.99   \\
 \bottomrule 
\end{tabular}  \label{ch4ex1:stdFEM_t=0.5}
\end{table}

 The exact solution and the two-grid solution in Algorithm \ref{al-FI} when $h=1/64, t=0.5$ are shown in Fig. \ref{exact_twogrid_phi_pn2H}, \ref{exact_twogrid_p1_pn2H} and \ref{exact_twogrid_p2_pn2H}.
 Compared the exact solution (a) and the two-grid solution (b), we can easily find that the two-grid finite element solution and the exact one are similar, which indicates the validity of the numerical test.

\begin{figure}[H]
 \centerline{
 \includegraphics[height=8.5cm, width=17.0cm]{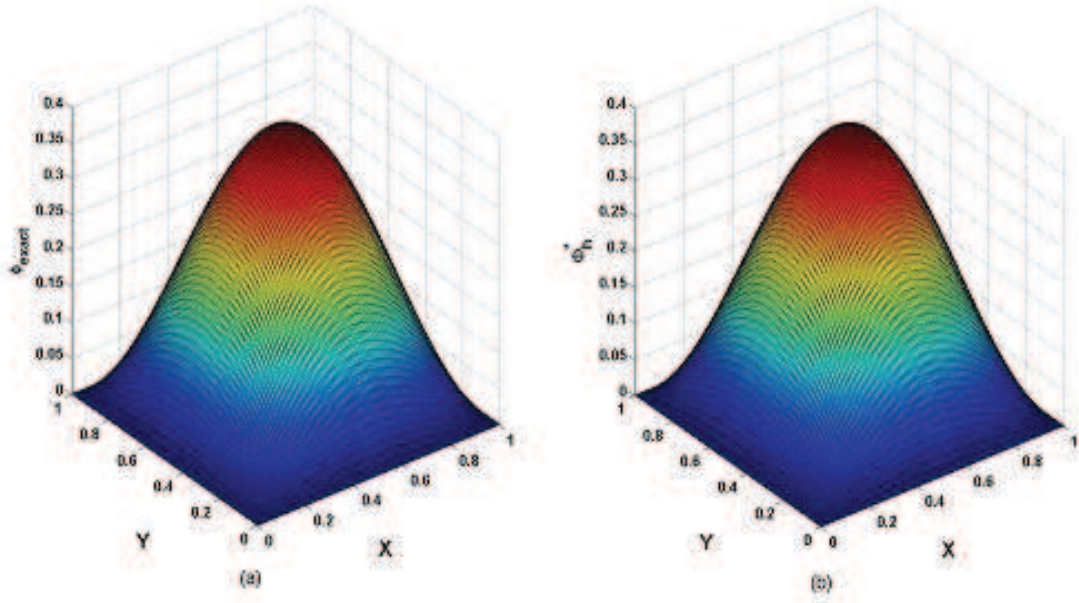} }
 \caption{ The exact solution (a) and two-grid solution (b) of $\phi$: $h=1/64, \tau = h^2, t=0.5$. }
 \label{exact_twogrid_phi_pn2H}
\end{figure}

\begin{figure}[H]
 \centerline{
 \includegraphics[height=8.5cm, width=17.0cm]{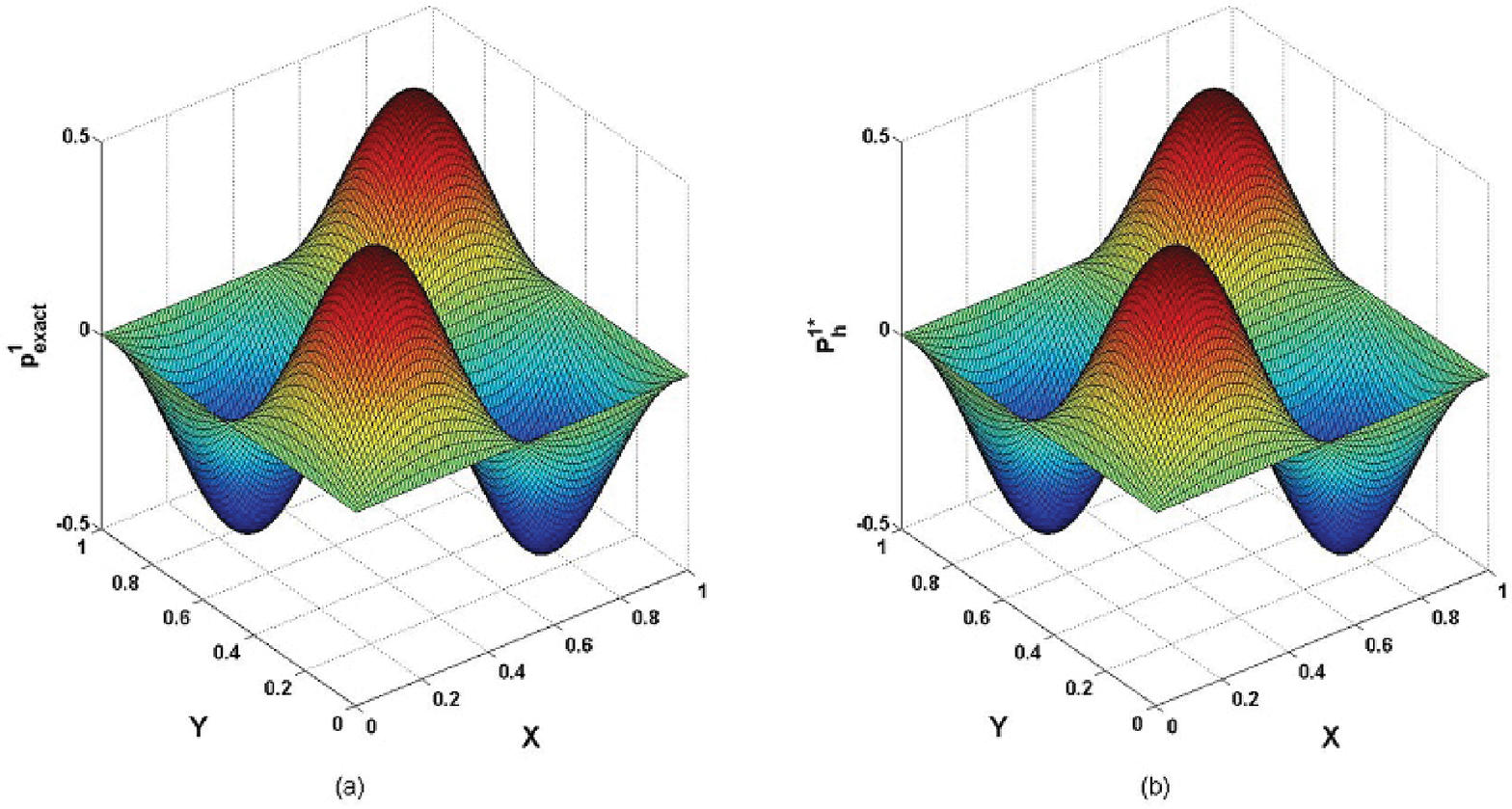} }
 \caption{ The exact solution (a) and two-grid solution (b) of $p^1$: $h=1/64, \tau = h^2, t=0.5$. }
 \label{exact_twogrid_p1_pn2H}
\end{figure}

 \begin{figure}[H]
 \centerline{
 \includegraphics[height=8.5cm, width=17.0cm]{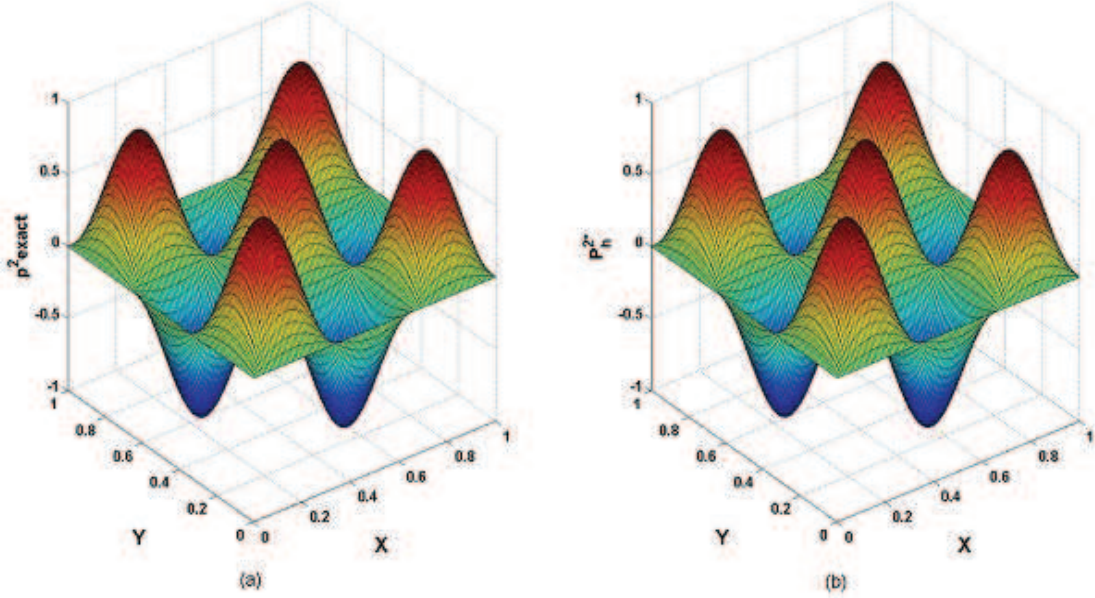} }
 \caption{ The exact solution (a) and two-grid solution (b) of $p^2$: $h=1/64, \tau = h^2, t=0.5$. }
 \label{exact_twogrid_p2_pn2H}
\end{figure}

 Table \ref{ch4ex1:TG_Sem_t=0.5} and Table \ref{ch4ex1:TG_Sem_L2_t=0.5} show the
 errors between the exact solution and the two-grid solution of Algorithm \ref{al-FI} with varying mesh size $H = \sqrt{h}$, where the order represents the convergence order relating to the fine grid size $h$ in $H^1$ or $L^2$ norm. The errors indicate that the numerical results coincide with the theoretical result in Theorem \ref{theo-F1} when $r=1$. For Algorithm \ref{aII-FII}, the errors between the exact solution and the two-grid solution with varying mesh size $H = \sqrt{h}$ is shown in Table \ref{ch4ex1:TG_Full_t=0.5}, where the order denotes the convergence order relating to the fine grid size $h$ in $H^1$ norm.

 \begin{table}[H]
 \centering
 \caption{$H^1$ error of the two-grid solutions of Algorithm \ref{al-FI} }
 \begin{tabular}{p{.8cm}p{1.6cm}p{2.1cm}p{1.5cm}p{2.1cm}p{1.5cm}p{2.1cm}p{.8cm}}
 \toprule 
  $H$ & $h=H^2$ & $\|\Phi_h^*-\phi\|_1$ & Order & $\|P_{h*}^{1}-p^1\|_1$
  & Order & $\|P_{h*}^{2}-p^2\|_1$ & Order  \\
  \midrule 
  1/3 & 1/9  & 1.5014E-01 & $-$  & 7.0997E-01 & $-$  & 2.6896E+00 & $-$ \\
  1/4 & 1/16 & 8.5654E-02 & 0.98 & 4.1360E-01 & 0.94 & 1.6093E+00 & 0.89 \\
  1/5 & 1/25 & 5.4812E-02 & 1.00 & 2.6605E-01 & 0.99 & 1.0447E+00 & 0.97 \\
  1/6 & 1/36 & 3.8125E-02 & 1.00 & 1.8547E-01 & 0.99 & 7.3026E-01 & 0.98 \\
  1/7 & 1/49 & 2.8006E-02 & 1.00 & 1.3635E-01 & 1.00 & 5.3766E-01 & 0.99 \\
  1/8 & 1/64 & 2.1450E-02 & 1.00 & 1.0448E-01 & 1.00 & 4.1223E-01 & 0.99 \\
  \bottomrule 
 \end{tabular} \label{ch4ex1:TG_Sem_t=0.5}
\end{table}

\begin{table}[H]
 \centering
 \caption{$L^2$ error of the two-grid solutions of Algorithm \ref{al-FI} }
 \begin{tabular}{p{.8cm}p{1.6cm}p{2.1cm}p{1.5cm}p{2.1cm}p{1.5cm}p{2.1cm}p{.8cm}}
 \toprule 
  $H$ & $h=H^2$ & $\|\Phi_h^*-\phi\|$ & Order & $\|P_{h*}^{1}-p^1\|$
  & Order & $\|P_{h*}^{2}-p^2\|$ & Order  \\
  \midrule 
  1/3 & 1/9  & 7.4133E-03 & $-$  & 3.2547E-02 & $-$  & 1.2124E-01 & $-$ \\
  1/4 & 1/16 & 2.4206E-03 & 1.95 & 1.0885E-02 & 1.90 & 4.2978E-02 & 1.80 \\
  1/5 & 1/25 & 9.9652E-04 & 1.99 & 4.5062E-03 & 1.98 & 1.8111E-02 & 1.94 \\
  1/6 & 1/36 & 4.8235E-04 & 1.99 & 2.1860E-03 & 1.98 & 8.8367E-03 & 1.97 \\
  1/7 & 1/49 & 2.6054E-04 & 2.00 & 1.1817E-03 & 2.00 & 4.7904E-03 & 1.99 \\
  1/8 & 1/64 & 1.5285E-04 & 2.00 & 6.9361E-04 & 2.00 & 2.8151E-03 & 1.99 \\
  \bottomrule 
 \end{tabular} \label{ch4ex1:TG_Sem_L2_t=0.5}
\end{table}

 Compared Table \ref{ch4ex1:TG_Sem_t=0.5} and Table \ref{ch4ex1:TG_Sem_L2_t=0.5} with Table \ref{ch4ex1:stdFEM_t=0.5} and Table \ref{ch4ex1:stdFEM_jishuL2_t=0.5}, respectively,  we can find that when $H = \mathcal{O}(h^{1/2})$, the errors in $H^1$ norm and $L^2$ norm approximate the first-order and the second-order, respectively, which indicates the solution of Algorithm \ref{al-FI} remains the same convergence order as the standard finite element method. Similarly, by comparing the results in
 Table \ref{ch4ex1:TG_Full_t=0.5} with that in
 Table \ref{ch4ex1:stdFEM_t=0.5}, the errors show that the full decoupled two-grid Algorithm \ref{aII-FII} can also achieve the same order of accuracy as the standard finite element method.
\begin{table}[H]
 \centering
 \caption{ $H^1$ error of the two-grid solutions of Algorithm \ref{aII-FII} }
 \begin{tabular}{p{.8cm}p{1.6cm}p{2.1cm}p{1.5cm}p{2.1cm}p{1.5cm}p{2.1cm}p{.8cm}}
 \toprule 
  $H$ & $h=H^2$ & $\|\Phi_h^*-\phi\|_1$ & Order & $\|P_{h*}^{1}-p^1\|_1$
  & Order & $\|P_{h*}^{2}-p^2\|_1$ & Order  \\
  \midrule 
 1/3 & 1/9  & 1.5014E-01 & $-$  & 7.0999E-01 & $-$  & 2.6896E+00 & $-$  \\
 1/4 & 1/16 & 8.5657E-02 & 0.98 & 4.1366E-01 & 0.94 & 1.6093E+00 & 0.89 \\
 1/5 & 1/25 & 5.4814E-02 & 1.00 & 2.6613E-01 & 0.99 & 1.0447E+00 & 0.97 \\
 1/6 & 1/36 & 3.8127E-02 & 1.00 & 1.8554E-01 & 0.99 & 7.3032E-01 & 0.98 \\
 1/7 & 1/49 & 2.8007E-02 & 1.00 & 1.3641E-01 & 1.00 & 5.3770E-01 & 0.99 \\
 1/8 & 1/64 & 2.1451E-02 & 1.00 & 1.0454E-01 & 1.00 & 4.1227E-01 & 0.99 \\
 \bottomrule 
 \end{tabular} \label{ch4ex1:TG_Full_t=0.5}
\end{table}

 The CPU time cost of Algorithm \ref{ALGO_num} (the finite element method combined with the Gummel iteration), Algorithm \ref{al-FI} and Algorithm \ref{aII-FII} are given in Table \ref{ch4ex1:CPU_time-To-0.5}, where the letter $h$ represents the size of grid in Algorithm \ref{ALGO_num} and also the size of the fine grid in Algorithm \ref{al-FI} and \ref{aII-FII}. As shown in Table \ref{ch4ex1:CPU_time-To-0.5}, the CPU time cost by Algorithm \ref{al-FI} or \ref{aII-FII} is much less than that of Algorithm \ref{ALGO_num} as $h$ becomes small, which reveals that the two-grid method is more efficient than the finite element method combined with the Gummel iteration. Moreover, Algorithm \ref{aII-FII} could achieve a better effect for large scale problems if a parallel program is applied at each time level.
\begin{table}[H]
\centering
\caption{ The total CPU time (second) (Example 5.1)}
\begin{tabular}{p{3.2cm}p{4.5cm}p{4.5cm}p{2.cm}}
\toprule 
  $h$ & \tabincell{l}{Algorithm \ref{ALGO_num}\\ CPU Time} &
  \tabincell{l}{Algorithm \ref{al-FI} \\CPU Time}&
  \tabincell{l}{Algorithm \ref{aII-FII}\\ CPU Time} \\
  \midrule 
 1/9  & 4.35     & 1.19   & 0.62    \\
 1/16 & 37.62    & 7.83   & 6.19    \\
 1/25 & 351.76   & 29.52  & 22.89   \\
 1/36 & 3423.40  & 99.10  & 84.28   \\
 1/49 & 8954.83  & 253.18 & 225.79  \\
 1/64 & 28221.89 & 666.44 & 603.31  \\
 \bottomrule 
\end{tabular}  \label{ch4ex1:CPU_time-To-0.5}
\end{table}

 \noindent
 \emph{Example 5.2}
 We consider the following PNP model for simulating asymmetrical conductance changes in Gramicidin A (gA) with two ion species in a $1:1$ CsCl solution with valence $+1$ and $-1$, respectively,
 \begin{eqnarray} \label{ion-pnp}
  \left\{ \begin{array}{l}
  \frac{\partial p}{\partial t} = \nabla\cdot D_p\big(\nabla p
   +\frac{e}{K_BT} p\nabla\phi\big), \quad~~~\mbox{in} ~~{\Omega_s}, \\
   \\
  \frac{\partial n}{\partial t} = \nabla\cdot D_n\big(\nabla n
   -\frac{e}{K_BT} n\nabla\phi\big), \quad~~\mbox{in} ~~{\Omega_s}, \\
   \\
  -\nabla\cdot(\varepsilon\nabla\phi) = (p-n)e, \quad~~~~~~~~~~~~\mbox{in} ~~\Omega,
 \end{array} \right.
\end{eqnarray}
 where $\Omega=\Omega_s\cup \Omega_m$, $\Omega_s$ is the solvent region, $\Omega_m$ is the solute region, $\phi$ is the electrostatic penitential,
 $p(x)$ and $n(x)$ are the concentrations of the positive ions and the negative ions in the bulk solvent respectively. The constant coefficients $D_p$ and $D_n$ are the diffusion coefficients of the positive ions and the negative ions respectively, $K_BT$ is the Boltzmann energy constant, $e$ is the charge for one electron and
 $\varepsilon = \left\{
   \begin{array}{l}
             2\varepsilon_0, \ \ \mbox{in} \ \Omega_m,  \\
             80\varepsilon_0, \ \mbox{in} \ \Omega_s,
   \end{array} \right. $
 is the dielectric permittivity coefficient, where $\varepsilon_0$ is the dielectric constant of vacuum.

 Suppose $\Gamma_1$ and $\Gamma_2$ are the interfaces, where $\Gamma_1$ is the boundaries of membranes, $\Gamma_2$ is the boundaries of protein exposed to solvent, $\Gamma_3$ is the outside boundaries of $\Omega_s$ and the boundaries of the whole domain are denoted by $\partial\Omega$.  The meshes of the simulation box and boundaries are shown in Fig. \ref{ion-mesh2D}.

\begin{figure}[H]
 \centerline{
 ~~~~~~~~~~~~~~~~~\includegraphics[height=5.5cm, width=13.0cm]{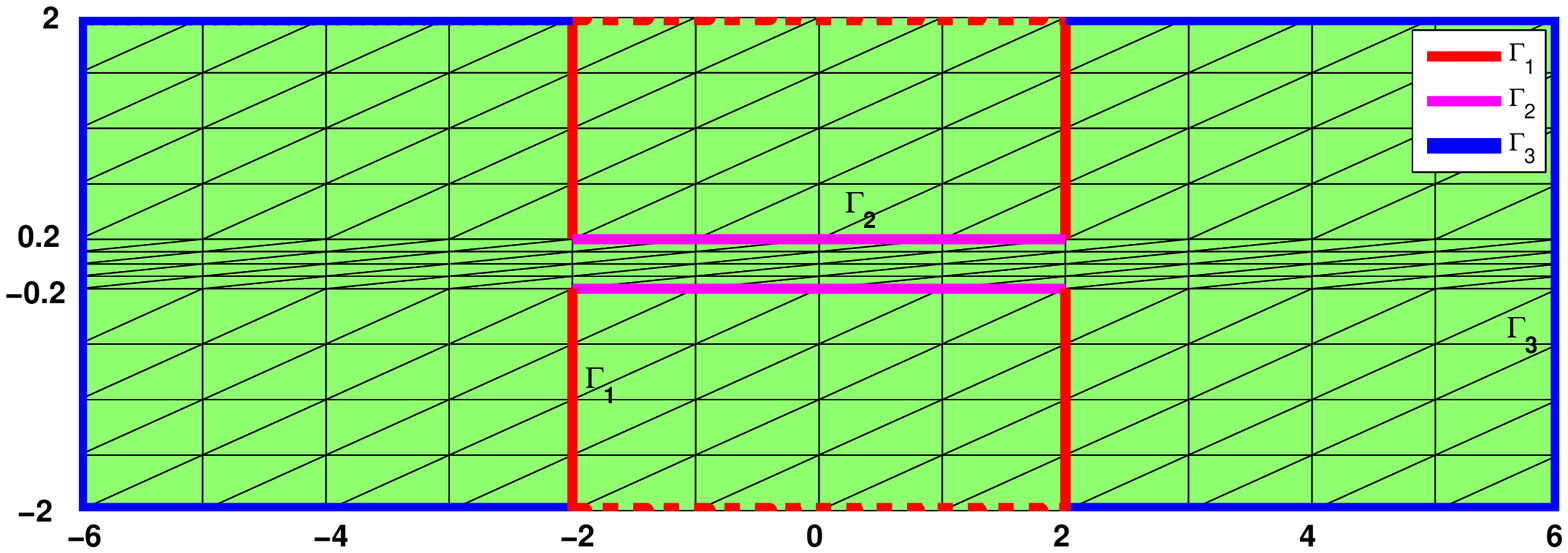} }
 \caption{ The meshes of the simulation box and boundaries of $\Omega_s$. The red solid lines, $\Gamma_1$, are the boundaries of membranes, the magenta lines, $\Gamma_2$, are the boundaries of protein exposed to solvent, and the blue lines, $\Gamma_3$, are the outside boundaries of $\Omega_s$. }
 \label{ion-mesh2D}
\end{figure}

 Then the boundary and initial conditions are described as follows
 \begin{eqnarray} \label{ion-initial-bd}
  \left\{\begin{array} {l}
   (D_p\nabla p+\frac{D_p e}{K_BT} p\nabla\phi)\cdot\nu = 0, \quad~  \mbox{on} \ \Gamma_1\cup \Gamma_2, \\
   (D_n\nabla n-\frac{D_n e}{K_BT} n\nabla\phi)\cdot\nu = 0, \quad  \mbox{on} \ \Gamma_1\cup \Gamma_2, \\
   p=p_\infty, ~ n=n_\infty, \quad ~~~~~~~~~~~~~~ \mbox{on} \ \Gamma_3,\\
   p(\cdot,0) = p_\infty,~ n(\cdot,0)= n_\infty, \\
   {[}\varepsilon\nabla\phi{]}=\rho_1, \quad ~~~~~~~~~~~~~~~~~~~~~ \ \mbox{on}\ \Gamma_1,\\
   {[}\varepsilon\nabla\phi{]}=\rho_2,  \quad~~~~~~~~~~~~~~~~~~~~~~ \mbox{on} \ \Gamma_2, \\
   \phi(x,t)=-\delta V x/L.  \quad~~~~~~~~~~~~~ \mbox{on} \ \partial\Omega,
  \end{array} \right.
 \end{eqnarray}
 where $\nu$ is the exterior unit normal with direction from solvent region to macromolecule part on the boundary, $\rho_1$ and $\rho_2$ are the charge densities on the surface of membranes and protein respectively, $\delta V$ is the voltage difference between the left and right edges of the box along $x$ direction, $L$ is the length of the simulation box, $p_\infty$ and $n_\infty$ are the initial-boundary charge densities.

 This example uses the similar setup as the model presented in \cite{S. Xu 2014}. Suppose $\Omega=\Omega_s\cup \Omega_m =[x,y]=[-6,6]\times [-2,2]$,  $\Omega_m=[-2,2]\times[-2,-0,2]\cup[-2,2]\times[0.2,2]$ denotes the solute region, $\Omega_I=[-2,2]\times[-0.2,0,2]$ is the ion channel region and $\Omega_s\setminus\Omega_I=[-6,-2]\times[-2,2]\cup [2,6]\times[-2,2]$ represents the solvent region excluding $\Omega_I$.
 In our computations, the values of the parameters mentioned above are reported in Table \ref{ion-pnp-parameters}.

\begin{table}[H]
 \centering
 \caption{ The parameters for PNP equations
          \eqref{ion-pnp}-\eqref{ion-initial-bd} }
  \small
 \begin{tabular}{p{3.6cm}p{3.2cm}p{3.8cm}p{3.5cm}}
 \toprule 
  Variables   & Values     & Variables      & Values  \\
  \midrule 
  Diffusion coefficient: $D_p$  & $2.0561\times10^{-9}m^2/s$ & Initial density: $p_\infty(n_\infty)$ & $6.02\times 10^{25}g/m^3$    \\
  Diffusion coefficient: $D_n$  & $2.0321\times10^{-9}m^2/s$ & Length of the box: $L$ & $1.0\times 10^{-10}m$ \\
  Boltzmann energy: $K_BT$      & $4.14\times 10^{-21} J$    & Permittivity of vacuum: $\varepsilon_0$ & $8.85\times 10^{-12}C^2/(N\cdot m^2)$      \\
  Elementary charge: $e$        & $1.6\times10^{-19} C$      &        \\
 \bottomrule 
 \end{tabular} \label{ion-pnp-parameters}
\end{table}

 In order to compute the finite element solution of \eqref{ion-pnp}, we first give the weak formulation as follows.
 Let ${V_0} = \{ \upsilon {\left| {\upsilon  \in H} \right.^1}({\Omega _s}),\upsilon \left| {_{\partial \Omega }} \right. = 0)$ and ${V_{0,{\Gamma _3}}} = \{ \upsilon {\left| {\upsilon  \in H} \right.^1}({\Omega _s}),\upsilon \left|_{\Gamma _3} \right. = 0)$.
 Find $\phi  \in {L^2}(0,T;V),~p,n\in {L^2}(0,T;{V_{{\Gamma _3}}}) \cap {L^\infty }({\Omega _T})$ such that
 \begin{eqnarray} \label{ion-w1}
   {\big(\frac{{\partial p}}{{\partial t}},\upsilon \big)_{{\Omega _s}}} + {\big({D_p}\nabla p + D_p\frac{e}{K_BT}p\nabla \phi ,\nabla \upsilon \big)_{{\Omega _s}}} = 0,\quad \forall \upsilon  \in {V_{0,\Gamma }}_{_3}, \\  \label{ion-w2}
   {\big(\frac{{\partial n}}{{\partial t}},\upsilon \big)_{{\Omega _s}}} + {\big({D_n}\nabla n - D_n\frac{e}{K_BT}n\nabla \phi ,\nabla \upsilon \big)_{{\Omega _s}}} = 0,\quad\forall \upsilon  \in {V_{0,{\Gamma _3}}}, \\   \label{ion-w3}
   {\big(\varepsilon \nabla \phi ,\nabla w \big)_\Omega }
   = {\big((p - n)e,w\big)_\Omega}+ {\big([\varepsilon\nabla\phi],w\big)_{\Gamma_1\cup\Gamma_2},\quad\forall w\in {V_0}}.
 \end{eqnarray}

 In this computation, the implicit Euler scheme is used for the time discretization with time step $\tau$. We set the final time $T=1.0$ and choose the time step $\tau= 0.1$. The bulk densities of CsCl solution is $0.1M$ and the voltage difference $\delta V = 8V$. The edge average finite element method (EAFEM) \cite{J.Xu-L.Z 1999} is used in our calculation to solve the density equations. For the charge distributions $p$ and $n$, the piecewise linear element is used on the triangulation of domain $\Omega_s$ and the second order isoparametric finite element (cf. \cite{P.G.Ciarlet 1978}) is used for the potential. The finite element approximation $(P_h, N_h, \Phi_h)$ satisfies
 \begin{eqnarray} \label{ion-wh1}
   {\Big(\frac{P_h^{n+1}-P_h^n}{\tau},\upsilon_h \Big)_{{\Omega_s}}} + {\Big({D_p}\nabla P_h^{n+1} + D_p\frac{e}{K_BT}P_h^{n+1}\nabla\Phi_h^{n+1},\nabla\upsilon_h \Big)_{{\Omega_s}}} = 0, \\
   \label{ion-wh2}
   {\Big(\frac{N_h^{n+1}-N_h^n}{\tau},\upsilon_h \Big)_{{\Omega _s}}} + {\Big({D_n}\nabla N_h^{n+1} - D_n\frac{e}{K_BT}N_h^{n+1}\nabla\Phi_h^{n+1} ,\nabla\upsilon_h \Big)_{{\Omega _s}}} = 0, \\
   \label{ion-wh3}
   {\Big(\varepsilon\nabla\Phi_h^{n+1} ,\nabla w_h \Big)_\Omega }
   = {\Big((P_h^{n+1} - N_h^{n+1})e,w_h\Big)_\Omega}+ {\Big([\varepsilon\nabla\Phi_h^{n+1}],w_h\Big)_{\Gamma_1\cup\Gamma_2}}.
 \end{eqnarray}

 To illustrate the efficiency and effectiveness of the two-grid method for the ion channel problem, we first obtain the finite element solution of \eqref{ion-wh1}-\eqref{ion-wh3} by using EAFEM combined with the Gummel iteration. Then Algorithm \ref{al-FI} are used to solve \eqref{ion-wh1}-\eqref{ion-wh3} to get the two-grid solution $(P_h^*, N_h^*, \Phi_h^*)$. Both the accuracy of these two solutions and the CPU time costs of the two methods are compared.

All the computations are implemented on quasiuniform triangular meshes, see e.g. Fig. \ref{ion-mesh2D}. To obtain the convergence rate, we refine the initial mesh step by step uniformly in the solvent region $\Omega_s$ and the solute region $\Omega_m$, respectively. Since Example 5.2 is a problem without an analytic solution, we choose the finite element solution with the degrees of freedom $K = 115713$ as ``the exact solution" for the charge distributions $p,~n$, and the finite element solution with the degrees of freedom $K = 148225 $ as ``the exact solution" for the potential $\phi$, since they are defined in different domains.


 Here, we first define the discrete $L^2$ norm as follows:
 $$\|e\|_{L^2}=\sqrt{\frac{1}{K}\sum\limits_{i=1}^K |e_i|^2}, $$
 where $e = (e_1, e_2, \cdots, e_K)^T$. Denote $K_H$ and $K_h$ are the degrees of freedom on the coarse grid and the fine grid, respectively. The numerical results for the finite element solutions and the two-grid solutions are shown in Table \ref{ion-channel-ph-nh}--\ref{ion-CPU_time}. First compared Table \ref{ion-channel-ph-nh} with Table \ref{ion-channel-Tp-Tn}, the results show that the two-grid solutions have the similar order of accuracy as the finite element solutions for both the charge distributions $p, n$ and the electrostatic potential $\phi$, which indicates that this decoupling method is efficient for the PNP system describing the ion channel. Second, as shown in Table \ref{ion-CPU_time}, the CPU time cost by Algorithm \ref{al-FI} is much less than that of EAFEM as the degree of freedom becomes large, which indicates the efficiency of Algorithm \ref{al-FI}.  We also note that the accuracy of order in Table \ref{ion-channel-ph-nh} or \ref{ion-channel-Tp-Tn} is not so good as that in Example 5.1, since there are some charges on the interface of membranes which leads to the singularity of the solution for the PNP system in this example. The results can be improved if a better mesh could be used. We shall study the two-grid method on the ununiform meshes such as the adaptive mesh in our further work.


\begin{table}[H]
 \centering
 \caption{ $L^2$ error of the EAFEM for $P_h, N_h$ and $\Phi_h$ }
 \begin{tabular}{p{1.51cm}p{2.13cm}p{1.5cm}p{2.15cm}p{0.9cm}|p{1.51cm}p{2.13cm}p{.7cm}}
 \hline
    $K_h$ & $\|P_h-p\|_{L^2}$ & Order & $\|N_h-n\|_{L^2}$  & Order & $K_h$ &$\|\Phi_h-\phi\|_{L^2}$ &order \\
  \hline
  45   & 0.0288 & $-$   &10.6052 & $-$   &  49   & 5.7787 & $-$\\
  145  & 0.0443 &-0.736 & 8.1898 & 0.442 &  169  & 4.3420 & 0.462    \\
  513  & 0.0407 & 0.134 & 7.2571 & 0.191 &  625  & 3.2902 & 0.424    \\
  1921 & 0.0410 &-0.011 & 5.3555 & 0.460 &  2401 & 2.1120 & 0.659    \\
  7425 & 0.0262 & 0.662 & 3.3710 & 0.685 &  9409 & 1.2527 & 0.765    \\
  29185& 0.0098 & 1.437 & 1.2608 & 1.437 &  37249& 0.4599 & 1.457    \\
 \hline
 \end{tabular} \label{ion-channel-ph-nh}
\end{table}

\begin{table}[H]
 \centering
 \caption{ $L^2$ error of the two-grid method for $P_h^*, N_h^*$ and $\Phi_h^*$ }
 \begin{tabular}{p{0.95cm}p{1.cm}p{1.6cm}p{1.cm}p{1.8cm}p{0.9cm}|p{0.95cm}p{1.cm}p{1.8cm}p{.7cm}}
 \hline
  $K_H$ &$K_h$ &$\|P_h^*-p\|_{L^2}$ &Order& $\|N_h^*-n\|_{L^2}$ & Order& $K_H$ &$K_h$ &$\|\Phi_h^*-\phi\|_{L^2}$ &order \\
  \hline
  16   &45    &0.0252 & $-$   &10.6769 & $-$   & 16   &49    & 5.8649 & $-$     \\
  45   &145   &0.0437 &-0.941 & 8.2322 & 0.444 & 49   &169   & 4.4391 & 0.450   \\
  145  &513   &0.0402 & 0.132 & 7.2697 & 0.197 & 169  &625   & 3.3548 & 0.428   \\
  513  &1921  &0.0409 &-0.026 & 5.3668 & 0.460 & 625  &2401  & 2.1697 & 0.648    \\
  1921 &7425  &0.0266 & 0.636 & 3.3843 & 0.682 & 2401 &9409  & 1.2930 & 0.757   \\
  7425 &29185 &0.0103 & 1.386 & 1.2672 & 1.435 & 9409 &37249 & 0.4930 & 1.403   \\
 \hline
 \end{tabular} \label{ion-channel-Tp-Tn}
\end{table}

\begin{table}[H]
 \parbox{.355\textwidth}{
 \caption{ The CPU time (second) of the EAFEM and two-grid method (Example 5.2) }  \label{ion-CPU_time}   }
 \hspace{\fill}
  \parbox{.58\textwidth}{
  \centering
  \begin{tabular} {p{2.5cm}p{3cm}p{2.5cm}}
  \toprule 
  $K_h$ & \tabincell{l}{EAFEM} &
  \tabincell{l}{ two-grid method \\ (Algorithm \ref{al-FI} )} \\
  \midrule 
  49    & 2.499    & 2.812    \\
  169   & 4.298    & 4.313    \\
  625   & 11.830   & 9.173    \\
  2401  & 47.584   & 31.581   \\
  9409  & 211.146  & 133.373   \\
  37249 & 1060.000 & 631.299   \\
  \bottomrule 
  \end{tabular} }
\end{table}

\section{Conclusion} \label{sec-conclusion}

\noindent
 \par In this paper, we first give the optimal error estimate in $L^2$ norm with linear element for both semi- and fully discrete finite element approximation for the time-dependent Poisson-Nernst-Planck equations. Then the decoupling two-grid finite element algorithms are proposed for the time-dependent Poisson-Nernst-Planck equations. The optimal error estimates are obtained  for the electrostatic potential and the concentrations in $H^1$ norm. The numerical experiments show that the two-grid algorithms remain the same order of accuracy but cost much less computational time compared with the finite element method combined with the Gummel iteration. It is promising to extend this method to more complex PNP models, such as PNP equations for three dimensional ion channel and semiconductor devices, as well as modified PNP equations with size effects. \\
\\


\noindent
 {\bf Acknowledgement }
 The authors would like to thank Dr. Chunshen Feng and Dr. Shixin Xu for their  valuable discussions on numerical experiments. S. Shu was supported by the China NSF (NSFC 11571293). Y. Yang was supported by the China NSF (NSFC 11561016, NSFC 11661027, NSFC 11561015), Guangxi Colleges and Universities Key Laboratory of Data Analysis and Computation open fund and Guangxi Key Laboratory of Cryptogriaphy and information Security. B. Z. Lu was supported by Science Challenge Program under grant number TZ2016003, and China NSF (NSFC 21573274, 11771435). R. G. Shen was supported by Postgraduate Scientific Research and Innovation Fund of the Hunan Provincial Education Department (CX2017B268).   \\


\end{document}